\documentclass[a4paper,10pt]{article}
\usepackage{amssymb}
\usepackage{amsmath}
\usepackage{amsthm}
\usepackage{authblk}
\usepackage[all]{xy}
\usepackage{enumerate}
\usepackage{tikz}
\usepackage{mathrsfs}

\begin{document}

\title{\textbf{A Note on Graph Characteristics and Hadwiger's Conjecture}}
\author{Hanbaek Lyu}  
\affil{Seoul National University}

\maketitle

\theoremstyle{definition}	\newtheorem{problem}{Problem}
\theoremstyle{definition}	\newtheorem{proposition}{Proposition}[subsection]
\theoremstyle{definition}	\newtheorem{lemma}{Lemma}[subsection]
\theoremstyle{definition}	\newtheorem{corollary}{Corollary}[subsection]
\theoremstyle{definition}	\newtheorem{definition}{Definition}[subsection]
\theoremstyle{definition}	\newtheorem{theorem}{Theorem}[subsection]
\theoremstyle{definition}	\newtheorem{remark}{Remark}[subsection]
\theoremstyle{definition}	\newtheorem*{introduction}{Introduction}
\theoremstyle{definition}	\newtheorem{ex}{Example}[subsection]

\begin{abstract}
	This is a note on three graph parameters motivated by the Euler-Poincare characteristic for simplicial complex. We show those three graph parameters of a given connected graph $G$ is greater than or equal to that of the complete graph with $\max(h(G),\chi(G))$ vertices. This will yield three different simultaneous upperbounds of both the hadwiger number and chromatic number by means of the number of particular types of induced subgraphs. Some applications to Hadwiger's Conjecture is also discussed.
\end{abstract}

\section{Introduction}

The chromatic number of a graph is the least number of colours to colour the vertices of the given graph, such that no two adjacent vertices get the same colour. This maybe one of the most interesting quantity that can be defined on a graph, since a colouring only needs to satisfy the local criterion, yet the least possible number of colours seems to emerge from the global structure of the graph. The question that what criterion or structure of the graphs determines, or affectes, the chromatic number  has been a richful source of many interesting problems in graph theory. That the planarity of the graphs forces the chromatic number be no more than 4 is the famous four colour theorem, which was proved by Apel and Haken using discharging methods, and simplyfied by Robertson, Sanders, Seymour, and Thomas [8] later. However, these proofs used an assistence of computer, and there are still lots of mathematicians trying to find a simple theoretical proof that does not need an aid of computer. On the other hand, there is a famous conjecture proposed by Hugo Hadwiger [6] in 1943,  which states that the chromatic number is no greater than the hadwiger number of the graph, the maximum number $t$ such that the graph contains $K^{t}$ minor. 

To study the chromatic number and hadwiger number of graphs, we consider a function $\beth$ called the \textit{graph characteristic}, defined to be the function which assigns a number for each simple graph $G$ that decreases under arbitrary edge contractions and certain graph homomorphisms which are strong enough to yield a complete graph when applied to any given graph. The behaviour of the graph characteristic with edge contractions would give us some informations on the hadwiger number, while that with graph homomorphisms on the chromatic number of graphs. Each graph characteristic then gives us a similtaneous upper bound for the hadwiger number $h(G)$ and the chromaric number $\chi(G)$, and when $h(G)$ equals the upper bound, the Hadwiger's conjecture is true. 

In this paper, we discuss three different graph characteristics which are strongly motivated by the Euler-Poincare characteristic for simplicial complex, which is an alternating sum of the number of cells of each dimension. Indeed, graph itself has cell structure with the vertices as 0-cells and the edges as 1-cells. Then the Euler-Poincare characteristic for a graph $G=(V,E)$ with this CW-complex sturcture would be $|E|-|V|$,(we take the negative of the usual Euler-Poincare characteristic in order to make it decrease under the graph operations mentioned before.) and this function indeed has the property required to be a graph characteristic. We call this the \textit{first graph characteristic} and study in the section 2. 

However, the number of vertices and edges are too simplified to capture the structural information of graphs; they lack the information about how the vertices are connected to each other. To extend the first graph characteristic, we need to define the 2-cells in graphs. There is, indeed, a standard way to define 2-cells in a graph. Consider the famous Euler characterisric for graphs drawn in the surface of genus $g$. Let $F$ be the number of regions in the plane determined by the planar drawing of the graph. Then that $F-|E|+|V|=2-2g$ holds is a well-known fact. Here, in fact, we are veiwing the regions as the 2-cells of the graphs, and the Euler characteristic is then the Euler-Poincare characteristic of the 2-dimensional cell complex. In this point of view, however, the Euler characteristic is not defined on the graph itself, but on the underlying surface with the cell division given by the graph. Moreover, the Euler characteristic does not behave well with the graph homomorphisms. Hence, instead of the regions, we consider the induced cycles of a given graph $G$ as the 2-cells in $G$. Note that the previous CW-structure was also a simplicial complex structure, but not anymore. Then the Euler-Poincare characteristic becomes $|C|-|E|+|V|$, where $|C|$ denotes the number of induced cycles, which we call \textit{the second graph characteristic}, studied in section 3. 

In section 4, we go one dimension higher; we define the 3-cells in a graph $G$ as the induced subgraphs isomorphic to one of the certain four types of graphs called \textit{solid graphs}, and call them \textit{the solids in $G$}. The graph charcteristic $|S|-|C|+|E|-|V|$, called the \textit{the third graph characteristic}, is studied in section 4, where $|S|$ denote the number of solids in $G$. Some applications of the graph characteristic theroy to Hadwiger conjecture are discussed in section 5. 

\subsection{Preliminaries}

All graphs in this paper are assumed to be finite and simple. If $H$ is a subgraph of $G$ and $v\in V(G)\setminus V(H)$, then the graph $H+v$ is defined to be the subgraph of $G$ obtained by adding the new vertex $v$ to $H$ together with all edges incident to $v$ in $G$. Similarly, if $H'$ is a sugraph of $G$, then $H+H'$ is defined by the subgraph of $G$ obtained from $H\cup H'$ by adding all edges in $G$ between $H$ and $H'$. If $e\in E(G)\setminus E(H)$ and $H$ contains the two endpoints of $e$, then $H+e$ denotes the subgraph of $G$ obtained by adding the new edge $e$ to $H$. We define $H/e$ by $H$ if $e\notin E(H)$ and $H/e$ otherwise. If $v\in V(H)$ and $H$ is a subgraph of $G$. Then $\Gamma_{H}(v)$ denotes the set of all neighbors of $v$ in $H$. 

In this paper, we focus on two fundamental graph operations; edge and nonedge contraction. Let $e=xy$ be an edge of a simple graph $G=(V,E)$. The graph $G/e$ obtained from $G$ by contracting the edge $e$ is the simple graph with vertex set $V(G/e)=V(G)\sqcup \{v_{e}\}\setminus \{u,v\}$ and edge set 
\begin{equation*}
	E(G/e):=\{uv\in E \,|\, \{u,v\}\cap \{x,y\}=\emptyset \} \cup \{v_{e}w \,|\, \text{$xw\in E\setminus \{e\}$ or $yw\in E\setminus \{e\}$ }\}.
\end{equation*}
That is, the edge $e$ and its two end points $x,y$ are contracted to a single vertex $v_{e}$ with induced adjacency. If $uv$ is a nonedge of a graph $G$, then the graph obtained by \textit{nonedge contraction of $uv$} from $G$ is defined to be $(G+uv)/uv$. Hence nonedge contraction is the composition of adding the edge $uv$ and then contracting $uv$. It should be noted that there are alternative definition for edge and nonedge contraction which allows multiple edges or loops. But we are not allowing the edge and nonedge contraction to generate any multiple edges or loops. For example, if we contract one edge $e$ of a triangle $K^{3}$, then $K^{3}/e=K^{2}$, the complete graph with two vertices. 

Let $G$ be a graph. A graph $H$ is a \textit{minor} of $G$ if it can be obtained by zero or more edge contractions from a subgraph of $G$. Equivalentely, $H$ is a minor of $G$ if it can be obtained by minor operations : edge deletion, vertex deletion, and edge contraction. But as long as we consider $G$ connected and its complete graph minor, we only need the edge contractions. That is, 

\begin{proposition}
	If $G$ is a connected graph and $K$ is a complete graph minor of $G$, then there is a sequence of edge contractions from $G$ to $K$.
\end{proposition}

\begin{proof}
Let $H$ be the subgraph of $G$ from which the complete graph $K$ can be obtained by edge contractions. Write $K=H/e_{1}/\cdots/e_{l}$. Since $G$ is connected, one can contract $G$ onto the vertex set of $H$. That is, if we denote $H_{1},\cdots,H_{k}$ the connected component of $G\setminus H$, then one can contract each $H_{i}$ to a single vertex, say, $v_{i}$. Then applying the same edge contractions to $G$, we get $H':=H+v_{1}+\cdots+v_{k}$ as the resulting graph. Note that $H'$ is connected, so that each vertex $v_{i}$ has some edge $f_{i}\in E(H')$ incident to it, not necessarily distinct. Write $H'':=H'/f_{1}/\cdots/f_{k}$. Notice that $H''$ can be obtained from $H$ by adding some edges, and hecne $E(H)\subseteq E(H'')$. Then we have $K=H''/e_{1}/\cdots/e_{l}$, and therefore we have obtained $K$ from $G$ only by edge contractions. 
\end{proof}

Though the chromatic number of a graph is defined by means of vertex colouring, there is another formulation of it using the notion of nonedge contraction. Let $G=(V,E)$ be a simple connected graph with a nonedge $uv$. Then it is easy to see that we can extend an arbitrary vertex colouring $c$ of $G/uv$ to a vertex colouring $\overline{c}$ of $G$, by defining $\overline{c}(x):=c(x)$ if $x\notin \{u,v\}$ and $\overline{c}(x)=c(u)$ otherwise. Hence $\chi(G/uv)\ge \chi(G)$. On the other hand, if $G$ is not a complete graph and $c$ is a vertex colouring of $G$, then $c$ must colour some two nonadjacent vertices $u,v$ of $G$ by the same colour, since we need less than $|V|$ colours to colour $|V| $ vertices. Then $c$ induces the colouring $c'$ of $G/uv$, where $c'(x):=c(x)$ if $x\ne v_{uv}$, and $c'(x)=c(u)$ otherwise; yields $\chi(G/uv)\le \chi(G)$. Thus, we conclude that if $G$ is non-complete, then there is a nonedge $uv$ of $G$ such that $\chi(G/uv)=\chi(G)$. Then by induction, we have that 

\begin{proposition}
	Let $G=(V,E)$ be a graph. Then there is a sequence of nonedge contractions from $G$ to $K^{\chi(G)}$. 
\end{proposition}

Hence, the chromatic number of a graph $G$ is the least possible size of the complete graph one can obtain from $G$ by successive nonedge contractions. More geometrically, think of a graph as a mobile made of magnetic vertices that pull each other, with some edges between them that never off any of the end vertices but can move freely in any angle. Suppose that no two adjacent vertices pull each other. Now one can gently "fold" some part of it so that any two non-adjacent vertices can be merged. Then one gets a complete graph when there is no more applicable folding, and the size of the smallest possible complete graph is the chromatic number of that graph. Note that the nonedge contraction is also called a vertex identification or a simple folding, and is an example of graph homomorphism.

\section{The First Graph Characteristic }

\subsection{The first graph characteristic and edge contractions}

As mensioned before, we consider the cannonical CW-sturcure on the simple graphs; 0-cells are the vertices, and 1-cells are edges. 

\begin{definition}
	Let $G=(V,E)$ be a connected graph. We define a function $\beth^{1}$ called \textit{the first graph characteristic} by 
\begin{equation*}
	\beth^{1}(G)=|E(G)|-|V(G)|.
\end{equation*}
\end{definition}

\begin{proposition}
	If $G=(V,E)$ is a connected graph with an edge $e=uv$, then $|E(G)|-|E(G/e)|=|C^{3}_{e}(G)|+1$.
\end{proposition}

\begin{proof}
	Since $G$ is simple, contraction of the edge $e$ can possibly make a set of double edges incident to $v_{e}$ that are to be deleted to be a single edge. Hence the difference $|E(G)|-|E(G/e)|$ equals the number of double edges in $G/e$ plus $1$ for the contracted edge $e$. Then observe that the each double edge in $G/e$ comes from a triangle in $G$ using the edge $e$. 
\end{proof}

\begin{theorem}
	If $G=(V,E)$ is a connected graph with an edge $e=uv$, then $\beth^{1}(G/e)\le \beth^{1}(G)$.
\end{theorem}

\begin{proof}
By proposition 2.1.1,
	\begin{eqnarray*}
		\beth^{1}(G/e) - \beth^{1}(G) \le -(|C^{3}_{e}(G)|+1)+1 \le 0.
	\end{eqnarray*}	
\end{proof}

\begin{corollary}
	If $G=(V,E)$ is a connected graph then $\beth^{1}(K^{h(G)})\le \beth^{1}(G)$. That is, 
\begin{equation*}
	\binom{h(G)}{2}-\binom{h(G)}{1}\le |E|-|V|.
\end{equation*}
\end{corollary}

\begin{proof}
	There is a sequence of succesive edge contractions from $G$ to $K^{h(G)}$, and each edge contraction decreases the first graph characteristic. Hence the assertion is clear. 
\end{proof}

\subsection{The first graph characteristic and nonedge contractions}

Let $D(G)$ be the set of length 2 paths in $G$, and $D_{uv}(G)\subset D(G)$ be the set of length 2 paths from $u$ to $v$. 

\begin{proposition}
	Let $G$ be a connected graph with nonedge $uv$. Then $\beth^{1}(G/uv)-\beth^{1}(G)=-|D_{uv}(G)|$.
\end{proposition}

\begin{proof}
	Let $w=(u+v)/uv\in V(G/uv)$. Then the difference $\beth^{1}(G/uv)-\beth^{1}(G)$ is the number of the double edges in $G/uv$ incident to the vertex $w$, and such number equals $|D_{uv}(G)|$.
\end{proof}

\begin{proposition}
	Let $G=(V,E)$ be a connected graph with order at least $2$, and suppose $G$ is not a complete graph. Then $G$ contains $D$ as an induced subgraph. In other words, there are two vertices of distance $2$ in $G$ if $G$ is not a complete graph.  
\end{proposition}

\begin{proof}
	A connected simple graph with order two is $K^{2}$. Hence we may assume $|V|\ge 3$. Since $G$ is not complete, there is a nonedge $uv$ in $G$. Then there is a shortest path $P$ from $u$ to $v$ with length $\ge 2$ since $G$ is connected. Let $v_{1},v_{2},v_{3}$ be three consecutive vertices on $P$. Note that $v_{1}v_{3}\notin E$ since otherwise $P$ would not be the smallest path from $u$ to $v$. Hence $D^{3}\subseteq G$.
\end{proof}

\begin{theorem}
	Let $G$ be a connected graph which is not a complete graph. Then there is a nonedge $uv$ of $G$ such that $\beth^{1}(G/uv)\le \beth^{1}(G)$.
\end{theorem}

\begin{proof}
By Proposition 2.2.2, there are some nonedge $uv$ of $G$ such that $|D^{3}_{uv}(G)|\ge 1$. Then Proposition 2.2.1 yields 
	\begin{eqnarray*}
		\beth^{1}(G/uv) - \beth^{1}(G) \le -|D_{uv}(G)|+1\le 0.
	\end{eqnarray*}	
\end{proof}

\begin{corollary}
	If $G=(V,E)$ is a connected graph then $\beth^{1}(K^{\chi(G)})\le \beth^{1}(G)$. That is, 
\begin{equation*}
	\binom{\chi(G)}{2}-\binom{\chi(G)}{1}\le |E|-|V|.
\end{equation*}
\end{corollary}

\begin{proof}
	Use induction on the number of vertices. If $|V|=1$, there is nothing to prove. Suppose $|V|>1$, and any connected graph $H$ such that $|V(H)|<|V(G)|$ satisfies the assertion. Now by Theorem 2.2.1, there is a nonedge $uv$ in $G$ such that $\beth^{1}(G/uv)\le \beth^{1}(G)$. Then $\chi(G)\le \chi(G/uv)$, and note that $\beth^{1}(K^{n})=\binom{n}{2}-n$ is an increasing function in $n$. Then since $G/uv$ is a connected graph with one less vertex, the induction hypothesis tells us that 
\begin{equation*}
	\beth^{1}(K^{\chi(G)})\le \beth^{1}(K^{\chi(G/uv)}) \le \beth^{1}(G/uv) \le \beth^{1}(G).
\end{equation*}
This completes the induction.
\end{proof}

Hence theorem 2.1.1. and 2.2.1 tells us that the first graph characteristic essentialy dicreases by the edge and nonedge contractions, and hence the resulting complete graphs $K^{h(G)}$ and $K^{\chi(G)}$ must have less values of $\beth^{1}$ then $G$. But since $\beth^{1}(K^{n})=\binom{n}{2}-n$ is an increasing function in $n$, we then have an upper bound both for $h(G)$ and $\chi(G)$ by means of $\beth^{1}(G)$. 

\begin{definition}
	For a connected graph $G=(V,E)$, the quantity $B(G)=\lfloor\frac{3+\sqrt{9+8(|E|-|V|)}}{2}\rfloor$ is called the \textit{first upper bound} of $G$.
\end{definition}

\begin{corollary}
	For a connected graph $G$,  both $h(G)$ and $\chi(G)$ are bounded above by $B(G)$. 
\end{corollary}
 
\begin{proof}
	Solving the quadratic inequality $\binom{N}{2}-N\le \beth^{1}(G)$ yields the assertion.
\end{proof}

\begin{remark}
Note that the first graph characteristic is equivalent to the rank of fundamental group of a graph. To see this, recall that the fundamental group of a connected graph $G$ is the free group generated by the edges not contained in a fixed maximal tree $T$ of $G$. Hence the number of generators of $\pi_{1}(G)$ equals $|E(G)|-|E(T)|=|E(G)|-(|V(G)|-1)=|E(G)|-|V(G)|+1$. Therefore $\pi_{1}(G)$ is the free group with $\beth^{1}(G)+1$ generators. 
\end{remark}

\section{The Second Graph Characteristic}

In this section, we consider the induced cycles of a given graph $G$ as the 2-cells and relate the number of them with the hadwiger number  and chromatic number of $G$. 

\subsection{The second graph characteristic and edge contractions}

\begin{definition}
	Let $G=(V,E)$ be a connected graph. We define a function $\beth^{2}$ which is called \textit{the second graph characteristic} by 
\begin{equation*}
	\beth^{2}(G)=|C(G)|-|E(G)|+|V(G)|,
\end{equation*}
where $C(G)$ is the set of all induced cycles in $G$. 
\end{definition}

\begin{proposition}
	Let $G$ be a connected graph with an edge $e=uv$. Then there is an injection $\psi : C(G/e)\rightarrow C(G)$ such that $\psi(C)/e=C$.
\end{proposition}

\begin{proof}
	Define $\psi(C)=C$ if $C$ does not use the vertex $v_{e}$. Suppose $C$ uses $v_{e}$. Let $x,y$ be the two endpoints of the induced path $C-v_{e}$ in $G/e$. Then both $x$ and $y$ are adjacent to at least one of $u$ and $v$ in $G$. If either $u$ or $v$ is adjacent to both $x$ and $y$, then define $\psi(C)=C-v_{e}+u$ or $\psi(C)=C-v_{e}+v$, and otherwise define $\psi(C)=C-v_{e}+u+v$. The property $\psi(C)/e=C$ follows from the definition, and it implies the injectivity of $\psi$.
\end{proof}

Recall that $C^{3}_{e}(G)$ is the set of triangles in $G$ that uses the edge $e$.

\begin{theorem}
	If $G=(V,E)$ is a connected graph with an edge $e=uv$, then $\beth^{2}(G/e)\le \beth^{2}(G)$.
\end{theorem}

\begin{proof}
 	Notice that no triangle $T$ in $G$ using the edge $e$ is not in the image of $\psi$, since $T/e\simeq K_{2}$ and there is no induced cycle in $G/e$ with two vertices. Then it follows that 
\begin{equation*}
	|C(G/e)| = |\text{im}\,\psi| \le |C(G)|-|C^{3}_{e}(G)|.
\end{equation*}
Now we finish the proof by observing that 
	\begin{eqnarray*}
		\beth^{2}(G) - \beth^{2}(G/e) \ge |C^{3}_{e}(G)| - (|C_{e}^{3}(G)|+1)+1 \ge 0.
	\end{eqnarray*}	
\end{proof}

\begin{corollary}
	If $G$ is a connected graph then $\beth^{2}(K^{h(G)})\le \beth^{2}(G)$. That is, 
\begin{equation*}
\binom{h(G)}{3}-\binom{h(G)}{2}+\binom{h(G)}{1} \le \beth^{2}(G). 
\end{equation*}
\end{corollary}

\begin{proof}
	Similar to the proof of Corollary 2.1.1.
\end{proof}

Now that we know the second graph characteristic decreases via edge contraction, we would like to know when the second characteristic does not change. This can be done by investigating the map $\psi$ in detail. Let $G=(V,E)$ be a graph with an edge $e=uv$. Divide the induced cycles in $G$ that are not triangles using $e$ into two classes, one of which consists of induced cycles in $G$ that remains an induced cycle in $G/e$, and the other its compliment. That is $C(G)$ is a disjoint union of $C^{3}_{e}(G)$, $\mathcal{C}:=\{C\in C(G)\,|\, C/e\in C(G/e)\}$, and $\mathcal{C}':=\{C\in C(G)\,|\, C/e\notin C(G/e)\}$. Note that the triangles of $C^{3}_{e}(G)$ are the cycles of $C(G)$ which are to be contracted to $K^{2}$ in $G/e$, and hence the cycles of the other two classes remain a cycle in $G/e$. But by the property of $\phi$, the class $\mathcal{C}'$ is disjoint from the image of $\psi$. It is easy to see that an induced cycle $C$ does not remain an induced cycle in $G/e$ if and only if there is a length 2 path in $G$ using the edge $e$ such that whose two endpoints are nonadjacent and vertices of $C$; for, if $C$ uses $e$, then $C/e$ does not have a crossing edge in $G/e$ and otherwise, the existence of the crossing edge of $C/e$ in $G/e$ is equivalent to the existence of such length 2 path in $G$. For each cycle $C$ of $\mathcal{C}'$, let us associate an induced subgraph $C+v$ if $C$ uses $u$ or $C+u$ if $C$ uses $v$. Let $S^{1}_{e}$ be the set of all such associated induced subgraphs of $G$. Then $|S^{1}_{e}|=\mathcal{C}'$.

On the other hand, let us give an equivalence relation on the class $\mathcal{C}$ by defining $C\sim C'$ if $C/e=C'/e$. Then by the definition of $\psi$, we see that a class $[C]$ has size 2 if and only if $C$ uses either $u$ or $v$ and both $C/e-v_{e}+u$ and $C/e-v_{e}+u$ are induced cycles in $G/e$, and otherwise $[C]$ consists only one cycle. Note that each equivalence class in $\mathcal{C}/\sim$ contains exactly one cycle which belongs to the image of $\psi$; hence the number of cycles of the class $\mathcal{C}'$ is the number of size 2 equivalence class in $\mathcal{C}/\sim$. Now associate each size $2$ class $[C]=\{C,C'\}$ with the induced subgraph $C+C'$ in $G$, which looks like a diamond ring. Denote the set of such associated induced subgraphs of $G$ by $S^{2}_{e}$. We then have $|S^{2}_{e}|=\mathcal{C}$. 

Hence we have counted the exact difference of the number of induced cycles in $G$ and $G/e$:

\begin{proposition}
	Let $G=(V,E)$ be a graph with an edge $e=uv$ and let $S_{1}$, $S_{2}$ be defined as above. Then 
\begin{equation*}
|C(G)|-|C(G/e)|=|S^{1}_{e}|+|S^{2}_{e}|+|C^{3}_{e}(G)|.
\end{equation*}
\end{proposition}

This enables us to calculate the difference of the second characteristic of $G$ and $G/e$.

\begin{proposition}
	Let $G=(V,E)$ be a graph with an edge $e=uv$. Then 
\begin{equation*}
	\beth^{2}(G)-\beth^{2}(G/e)=|S^{1}_{e}|+|S^{2}_{e}|.
\end{equation*}
\end{proposition}

\begin{proof}
	Follows from the proof of Theorem 3.1.1 and Proposition 3.1.2.
\end{proof}

\begin{remark}
	The elements of the set $S^{1}_{e}$ and $S^{2}_{e}$ are two special types of the \textit{solids} in $G$, \textit{pyramid} and \textit{trihedron}, which we may define in section 4.1. 
\end{remark}

Before we proceed to next section, I would like to note an interesting application of the construction of the map $\phi$. Let $l(G)$ denote the length of the largest cycle in the graph $G$. 

\begin{proposition}
	Let $G=(V,E)$ be a graph with an edge $e=uv$. Then we have $l(G/e)\le l(G)$.
\end{proposition}

\begin{proof}
	Let $C$ be the largest cycle in $G/e$. If $C$ does not use the vertex $v_{e}$, then $C\subseteq G$ and hence $l(G)\ge l(G/e)$. Otherwise, let $x,y$ be the two endpoints of the path $C-v_{e}$ in $G/e$. Both $x$ and $y$ are adjacent to at least one of $u$ and $v$, and at least one of $C-v_{e}+xu+yu$ and $C-v_{e}+xv+yv$ is a cycle in $G$ if either $u$ or $v$ is adjacent to both $x$ and $y$. If not, we can assume $xu,yv\in E$ and hence $C-v_{e}+xu+uv+vy$ is a cycle in $G$. Note that all the three cycles have length at least that of $C$. Therefore $l(G)\ge l(G/e)$. 
\end{proof}

\begin{proposition}
	Let $G=(V,E)$ be a graph. Then we have $h(G)\le l(G)$.
\end{proposition}

\begin{proof}
	We have $l(K^{h(G)})\le l(G)$ from Proposition 3.1.5. Clearly $l(K^{h(G)})=h(G)$.
\end{proof}

\subsection{The second graph characteristic and nonedge contractions}

A. D. Scott [9] related the induced cycles and chromatic number, by proving that for any pair of integers $k,l\ge 1$, there exists an integer $N(k,l)$ such that every graph with chromatic number at least $N(k,l)$ contains either $K_{k}$ or an induced odd cycle of length $\ge 5$ or an induced cycle of length $\ge l$. It roughly says that a graph with large chromatic number must contain a large complete graph or induced cycle. We are going to relate the induced cycles and chromatic number as well, but concenterating on the \textit{number} of induced cycles rather than their length. 

We have seen that in the previous section, we can always find two vertices of distance 2 in a noncomplete graph so that contracting of that nonedge decreases the first graph characteristic. For the second graph chracteristic, however, the same stretage does not works; such nonedge contraction could increase the number of induced cycles sometimes. Hence we seek for a \textit{better folding}, or better graph homomorphism, under which the number of induced cycle decreases. Following graph operation is one of such. 

\begin{definition}[vertex compression]
	Let $G=(V,E)$ be a graph with vertex $v$, and let $G_{v}$ be the induced subgraph of $G$ with vertex set $\Gamma_{G}(v)\cup \{v\}$. We denote by $G/v$ the graph obtained from $G$ by successive nonedge contractions on $G_{v}$ so that $G_{v}$ becomes $K_{v}:=K^{\chi(G_{v})}$ in $G/v$, and we call such operation the \textit{vertex compression of $v$}. 
\end{definition}

\begin{remark}
	A connected graph $G=(V,E)$ is a complete graph if and only if there is no further vertex compression. This is clear since $G$ is not a complete graph, then there are two vertices $u,v$ of distance 2 and any common neighbor $w$ of them gives a proper vertex compression. 
\end{remark}

\begin{definition}
	Let $G=(V,E)$ be a graph. The \textit{cone graph} $G^{w}$ is the graph obtained by adding a new vertex $w$ to $G$ with the edges $wv$ for all $v\in V$. The graph $G$ is called the \textit{base} of the cone graph $G^{w}$.
\end{definition}

\begin{proposition}
	Let $G^{w}$ be a cone graph. Then one has $ |C(G^{w})|=|C(G)|+|E(G)|$.
\end{proposition}

\begin{proof}
	There are $|C(G)|$ induced cycles in $G^{w}$ that does not use the vertex $w$. If $C$ is an induced cycle in $G^{w}$ that uses $w$, then it also uses two edges incident to $w$, say $wx$ and $wy$. If $xy\in E(G)$, then $C$ is the triangle with vertices $w,x$ and $y$. Otherwise, there should be another vertex $z\in G\setminus \{x,y\}$ that $C$ uses. Now since $wz\in E(G^{w})$ and $C\in C(G^{w})$, it follows that $C$ uses the three edges $wx,wy$ and $wz$, which contradicts the fact that every vertex in a cycle has degree $2$. This shows if $C$ is an induced cycle in $G^{w}$ that uses $w$, then $C$ must be a triangle which correspondes to an edge of $G^{w}$. Therefore there are $|E(G)|$ induced cycles in $G^{w}$ that uses $w$, and this shows the assertion. 
\end{proof}

\begin{proposition}
	 Let $G=(V,E)$ be a connected graph with a vertex $w$. Then there is an injection $\phi : C(G/w)\setminus C(K_{w}) \rightarrow C(G)\setminus C(G_{w})$ such that $\phi(C)/w-w=C$ for all $C\in C(G/w)$. 
\end{proposition}

\begin{proof}
	We may assume $G_{w}$ is not a complete graph, since otherwise $G/w=G$. Let $c$ be a proper colouring of $G/w$, so that the complete graph $K_{w}$ gets colours $\{1,2,\cdots,k\}$, with $c(w)=k$. We may extend $c$ to a colouring $\overline{c}$ of $G$, that is, $\overline{c}(v)=c(v/w)$. Let $C$ be an induced cycle in $G/w$ not contained in $K_{w}$. Observe that neither $C$ uses the vertex $w$ nor it contains more than three vertices of $K_{w}$ since otherwise it would be a triangle in $K_{w}$. Denote $H:=G_{w}-w$ and $K:=K_{w}-w$.
\begin{description}
	\item{Case 1.} $C$ uses no vertex $K_{w}$. 

$C$ is an induced cycle in $G$ and define $\phi(C)=C$.

	\item{Case 2.} $C$ uses one vertex $z$ of $K_{w}$. 

Note that $z$ cannot be $w$, since then $C$ must be a triangle contained in $K_{w}$. Hence we may suppose $c(z)=1$. Let $x,y$ be the two vertices of $C$ that are adjacent to $z$. Then $C-z$ is an induced path in $G/w$, and since $C-z\subseteq G/w\setminus K_{w}$, it is also a subgraph of $G$. Indeed, it must be an induced path in $G$, since otherwise it would not be an induced path in $G/v$. Then note that no vertex of $G_{w}$ of colour $1$ is adjacent to a vertex of $C-z$ that is neither $x$ or $y$. If there is a vertex $v$ of colour $1$ in $G_{w}$ that is adjacent to both $x$ and $y$, then define $\phi(C)=C-z+v$. If not, there are two vertices $u,v$ in $G_{w}$ each of which is adjacent to $x$ and $y$ repectively. Then $C-z+u+v$ is an induced path from $u$ to $v$ in $G$, and one observes that $C-z+u+v+w$ is an induced cycle in $G$; for, $w$ is adjacent only to $u$ and $v$ since $\Gamma(w)\cap V(C-z+u+v)=\{u,v\}$. In this case we define $\phi(C)=C-z+u+v+w$. 

	\item{Case 3.} $C$ uses two vertices $z_{1}$ and $z_{2}$ of $K_{w}$. 

If either $z_{1}$ or $z_{2}$ is $w$, then the two neighbors of $w$ in $C$ are in $K_{w}$, contradicting that $C$ uses only two vertices of $K_{w}$. Hence $z_{1},z_{2}\in V(K)$ and we may assume $\overline{c}(z_{1})=1$ and $\overline{c}(z_{2})=2$. Let $x,y$ be the neighbors of $z_{1},z_{2}$ in $C$ repectively. Then $x$ is adjacent to at least one vertex of colour 1 and $y$ is adjacent to some vertex of colour 2 in $H$. If there is a vertex $v$ in $K$ such that $\overline{c}(v)=1$ or 2 which is adjacent to both $x$ and $y$, then we define $\phi(C)=C-z+v$. Else if there are vertices $u,v$ such that $ux,uv,vy\in E(G)$ and $\{\overline{c}(u),\overline{c}(v)\}=\{1,2\}$, then we define $\phi(C)=C-z+u+v$. Otherwise, choose any vertices $v$ and $u$ of $H$ with $\overline{c}(u)=1$ and $\overline{c}(v)=2$ such that $ux,vy\in E(G)$. Then $u$ and $v$ are nonadjacent and hence $C-z+u+v+w$ is an induced cycle in $G$. In this case we define $\phi(C)=C-z+u+v+w$. 
\end{description}

Now we show $\phi$ is one to one. Observe that $\phi(C)/w-w=C$ in all cases. Hence $\phi(C)=\phi(C')$ yields $C=\phi(C)/w-w=\phi(C')/w-w=C'$ and thues $\phi$ is one to one. 
\end{proof}

\begin{remark}
	Let $C$ be an induced cycle in $G$. $C$ is called \textit{type 0} if $C/w$ uses no vertex of $K_{w}$, \textit{type 1} if $C/w$ is an induced cycle in $G/w$ which uses exactly one vertex of $K_{w}$, \textit{type 2} if $C/w$ is $C'+w$ where $C'$ is an induced cycle that uses exactly one vertex of $K$, \textit{type 3} if $C/w$ is an induced cycle in $G/w$ that uses exactly two vertices of $K$, and \textit{type 4} if the induced subgraph of $G/w$ with vertex set $V(C/w)$ is $C'+w$ where $C$ is an induced cycle in $G/w$ that uses two vertices of $K$. \\

Below are the graphs of $C/w$ in $G/w$ in corresponding five types. \\
\begin{equation*}
\begin{tikzpicture}
	\draw (-0.89,0) -- (0,1) -- (0.89,0);
	\filldraw [black] (0,1) circle (1.5pt); \coordinate [label=left:$w$]  (A) at (0,1);
	\draw (0,0) ellipse (25pt and 7pt);
	\coordinate [label=left:$K$]  (A) at (-1,0);
	\draw [ultra thick] (0,-1) circle (12pt);
	\coordinate [label=below:type 0]  (A) at (0,-1.5);
	
	\draw (1.61,0) -- (2.5,1) -- (3.39,0);
	\filldraw [black] (2.5,1) circle (1.5pt); \coordinate [label=left:$w$]  (A) at (2.5,1);
	\draw (2.5,0) ellipse (25pt and 7pt);
	\draw [ultra thick] (2.5,-0.4) circle (12pt);
	\filldraw [black] (2.5,0) circle (1.5pt); \coordinate [label=left:$z$]  (A) at (2.5,0.1);
	\coordinate [label=below:type 1]  (A) at (2.5,-1.5);

	\draw (4.11,0) -- (5,1) -- (5.89,0);
	\filldraw [black] (5,1) circle (1.5pt); \coordinate [label=left:$w$]  (A) at (5,1);
	\draw (5,0) ellipse (25pt and 7pt);
	\draw [ultra thick] (5,-0.4) circle (12pt);
	\filldraw [black] (5,0) circle (1.5pt); \coordinate [label=left:$z$]  (A) at (5,0.1);
	\draw[ultra thick] (5,1) -- (5,0);
	\coordinate [label=below:type 2]  (A) at (5,-1.5);
	\coordinate [label=below:Figure 1 : types of induced cycles]  (A) at (5,-2.2);

	\draw (6.61,0) -- (7.5,1) -- (8.39,0);
	\filldraw [black] (7.5,1) circle (1.5pt); \coordinate [label=left:$w$]  (A) at (7.5,1);
	\draw (7.5,0) ellipse (25pt and 7pt);
	\coordinate [label=below:type 3]  (A) at (7.5,-1.5);
	\filldraw [black] (7.2,0) circle (1.5pt);  \coordinate [label=left:$z_{1}$]  (A) at (7.2,0);
	\filldraw [black] (7.8,0) circle (1.5pt);  	\coordinate [label=right:$z_{2}$]  (A) at (7.8,0);
	\draw[ultra thick] (7.8,0) -- (8,-0.4) -- (7.85, -0.7) -- (7.5, -0.83) -- (7.15, -0.7) -- (7, -0.4) -- (7.2, 0) -- (7.8,0);

	\draw (9.11,0) -- (10,1) -- (10.89,0);
	\filldraw [black] (10,1) circle (1.5pt); \coordinate [label=left:$w$]  (A) at (10,1);
	\draw (10,0) ellipse (25pt and 7pt);
	\coordinate [label=below:type 4]  (A) at (10,-1.5);
	\filldraw [black] (9.7,0) circle (1.5pt);  \coordinate [label=left:$z_{1}$]  (A) at (9.7,0);
	\filldraw [black] (10.3,0) circle (1.5pt);  	\coordinate [label=right:$z_{2}$]  (A) at (10.3,0);
	\draw[ultra thick] (10.3,0) -- (10.5,-0.4) -- (10.35, -0.7) -- (10, -0.83) -- (9.65, -0.7) -- (9.5, -0.4) -- (9.7, 0) -- (10.3,0);
	\draw[ultra thick] (9.7,0) -- (10,0.97) -- (10.3,0);

\end{tikzpicture}.
\end{equation*}

\end{remark}

Though the following partition of induced cycles $C(G)\setminus C(G_{w})$ is not necessary to the second graph characteristic, it will be useful in the next section. 

\begin{proposition}
	In the setting of Proposition 3.2.2, an induced cycle $C\in C(G)\setminus C(G_{w})$ satisfies exactly one of the followings:
\begin{description}
	\item{(i)} $C$ is of the five types described in Remark 3.2.2.
	\item{(ii)} $C$ uses exactly two vertices of $H$, say $u,v$, and the two paths $P_{1},P_{2}$ from $u$ to $v$ along $C$ have no internal vertex in $G_{w}$.
	\item{(iii)} $C$ uses $\ge 3$ vertices of $H$ and does not use $w$. 
\end{description}
\end{proposition}

\begin{proof}
	First note that if $C$ is an induced cycle in $G$ not contained in $G_{w}$ that uses $\ge 3$ vertices of $H$, then it cannot use $w$, since otherwise $w$ has $\ge 3$ neighbors in $C$, which contradicts that any vertex of a cycle has exactly two neighbors. Hence if (iii) is not the case, then $C$ uses $\le 2$ vertices of $H$. If $C$ uses zero or one vertex of $H$, then it is of type 0 or 1. Hence we may assume $C$ uses exacly two vertices of $H$, say, $u$ and $v$. Let $P_{1},P_{2}$ be the two paths from $u$ to $v$ along $C$. Since $C$ is not contained in $G_{w}$, we can assume $P_{1}$ has no internal vertex in $G_{w}$. If we negate the case (ii), then $P_{2}$ becomes either a length 1 path from $u$ to $v$ or a length 2 path with $w$ middle vertex; the former corresponds to the type 3, and the latter corresponds to type 2 and 4, regarding the colours of $u$ and $v$ with the colouring $\overline{c}$. 
\end{proof}

\begin{lemma}
	Let $G=(V,E)$ be a connected graph with a vertex $w$. Then we have 
\begin{equation*}
|C(G)\setminus C(G_{w})|- |C(G/w)\setminus C(K_{w})| \ge |E(G)\setminus E(G_{w})|- |E(G/w)\setminus E(K_{w})|
\end{equation*}
\end{lemma}

\begin{proof}
	Note that the right hand side of the inequality of the assertion equals  the number of length $2$ paths $P$ in $G-w$ such that its two endpoints, say $x$ and $y$, are containd in $G_{w}-w$ and to be identified in the vertex compression of $w$. Let $A$ be the set of such paths $P$. Now we define a map $f:A\rightarrow C(G)$ by $f(P)=P+w$, which is clearly one to one. Note that $f(P)$ is an induced rectangle in $G$, and $f(P)/w$ is $K_{2}$. But observe that if an induced cycle $C$ in $G$ is in the range of the injection $\phi$ of Proposition 3.2.2, then $C/w$ must contain an induced cycle in $G/w$. Therefore the image of $f$ is disjoint from the image of $\phi$. That is, we have found $\text{im}\, f=|E(G/w)\setminus E(G_{w})|- |E(G)\setminus E(K_{w})|$ distinct induced cycles in $G$ which are not in the image of $\phi$. This shows the assertion. 
\end{proof}

\begin{theorem}
	Let $G=(V,E)$ be a connected graph with a vertex $w$. Then we have 
\begin{equation*}
	\beth^{2}(G/w) \le \beth^{2}(G).
\end{equation*}
\end{theorem}

\begin{proof}
	We use inductin on $|V|$. If $|V|=1$, there is nothing to prove. Suppose $|V|>1$. If $G_{w}$ is a complete graph, then there is nothing to prove since $G/w=G$. Now suppose $G_{w}$ is not a complete graph. Define $H:=G_{w}-w$ and $K:=K_{w}-w$, and write $G_{w}=H^{w}$ and $K_{w}=S^{w}$, the cone graphs with base $H$ and $K$. Then I claim that
\begin{equation*}
	|C(G_{w})|-|C({K_{w}})| \ge |E(H)| -|E(K)|. 
\end{equation*}
Note that $H$ may not be connected; let $H_{1},\cdots,H_{k}$ be the components of $H$. But since $|V(H)|<|V(G)|$, one can apply the induction hypothesis to each component of $H$, together with Theroem 2.2.1, to get
\begin{equation*}
	\Delta C \ge \Delta \beth^{1} \ge 0,
\end{equation*}
where the delta notation denotes the quantity of $G$ minus that of $G/w$, for intance, $\Delta C = |C(G)|-|C(G_{w})|$. Hence we apply vertex compression on each component $H_{i}$ repeatedly until $H_{i}$ becomes a complete graph $K_{i}$, while keeping the number of induced cycles decreasing. Hence we have corresponding complete graphs $K_{1},\cdots,K_{k}$ and $|C(K_{i})|\le |C(H_{i})|$. Now through nonedge contractions, identify all the complete graphs $K_{i}$ into the maximal one, say $K_{*}$. Note that $|C(K_{*})|=\max_{1\le i \le k}(|C(K_{i})|)\le \sum_{i=1}^{k}|C(H_{i})|=|C(H)|$. Then since vertex compressions are a composition of nonedge contractions, we get a sequence of nonedge contractions from $H$ to $K$ with $|C(K_{*})|\le |C(H)|$. It is clear that $|E(K_{*})|\le |E(H)|$. Hence Proposition $3.2.1$ yields 
\begin{equation*}
	|C(K_{*}^{w})|=|C(K_{*})|+|E(K_{*})|\le |C(H)|+|E(H)|=|C(H^{w})|. 
\end{equation*}
But since $K\subseteq K_{*}$ by the minimality of $K_{w}$, we get $|C(K)|\le |C(K_{*})|$, and hence $|C(K)|\le |C(H)|$. Therefore  we have shown the claim by
\begin{eqnarray*}
	|C(G_{w})|-|C(K_{w})| &=& |C(H)|-|C(K)| + |E(H)|-|E(K)| \\
	&\ge & |E(H)|-|E(K)|.
\end{eqnarray*}
Now observe that 
\begin{eqnarray*}
	|E(G_{w})|-|E(K_{w})|=|E(H)| -|E(K)|+|V(H)|-|V(K)|=|E(H)| -|E(K)|+\Delta V.
\end{eqnarray*}
Hence Lemma 3.2.1 and the claim imply
\begin{eqnarray*}
	\Delta C &=&  |C(G_{w})|-|C(K_{w})|+\left( |C(G)\setminus C(G_{w})|- |C(G/w)\setminus C(K_{w})| \right)\\
&\ge & |E(H)|-|E(K)|+|E(G)\setminus E(G_{w})|- |E(G/w)\setminus E(K_{w})|\\
&=& \Delta E - \Delta V
\end{eqnarray*}
and this completes the induction. 
\end{proof}

\begin{corollary}
	Let $G=(V,E)$ be a connected graph. Then $\beth^{2}(K^{\chi(G)})\le \beth^{2}(G)$. That is, 
\begin{equation*}
\binom{\chi(G)}{3}-\binom{\chi(G)}{2}+\binom{\chi(G)}{1} \le \beth^{2}(G)
\end{equation*}
\end{corollary}

\begin{proof}
	Follows from Theorem 3.2.1. The proof is similar to that of Corollary 2.2.1. 
\end{proof}

\begin{corollary}
	Let $G=(V,E)$ be a connected planar graph drawn in the plane and let $F$ be the number of regions. If $|C(G)|\le F+2$, then $\chi(G)\le 4$. 
\end{corollary}

\begin{proof}
	Suppose $|C(G)|\le F+2$. Then the Euler characteristic $F-|E|+|V|$ of the planar graph $G$ is $2$, so that $\beth^{2}(G)\le 4$. Note that $\beth^{2}(K^{n})=\binom{n}{3}-\binom{n}{2}+\binom{n}{1}$ is an increasing function and strictly increasing if $n\ge 3$. It is easy to check $\beth^{2}(K^{5})=5$. Then Collary 3.2.1 yields that $\chi(G)\le 4$. 
\end{proof}

Hence, we can estimate the hadwiger number and chromatic number by \textit{counting} the induced cycles, edges and vertices. For example, if $\beth^{2}(G)<57$ then $h(G),\chi(G)\le 8$, and $\beth^{2}(G)<36$ implies $h(G),\chi(G)\le 7$, and so on. 

\section{The Third Graph Characteristic}

In this section, we define a sort of 3-cell in graphs and extend our graph characteristic one dimension higher. 

\subsection{Solids and minimal closed surfaces}

\begin{definition}
	A graph consisting of a cycle $C$	and a vertex $w$ with at least three edges between $C$ and $w$ is called a \textit{pyramid}. Such vertex $w$ is called an \textit{apex} of the pyramid. A graph is called a \textit{trihedron} if it is a union of three paths $P_{1},P_{2},P_{3}$ that start and end at the same points such that each interior points of each paths are not used in any other path. A graph is called a \textit{stamp} if it is the union of a triangle $T$ and three paths $P_{1},P_{2},P_{3}$ of length $\ge 1$ such that the three paths starts from three different vertices of $T$ and ends at the same vertex $x$, and $P_{i}-x$ are vertex-disjoint. A graph $S$ is called a $\textit{prism}$ if $S=T\cup T' \cup P_{1}\cup P_{2}\cup P_{3}$ where $T,T'$ are vertex-disjoint triangles with vertices $x_{1},x_{2},x_{3}$ and $y_{1},y_{2},y_{3}$ respectively, and $P_{i}$s are vertex-disjoint paths from $x_{i}$ to $y_{i}$. The four graphs are described below. 

\begin{equation*}
\begin{tikzpicture}
	\filldraw [black] (0,-0.64) circle (1.5pt); 
	\filldraw [black] (0.6,-0.7) circle (1.5pt); 
	\filldraw [black] (1.03,-0.95) circle (1.5pt); 
	\filldraw [black] (0.7,-1.25) circle (1.5pt); 
	\filldraw [black] (0.1,-1.35) circle (1.5pt); 
	\filldraw [black] (-0.55,-1.3) circle (1.5pt); 
	\filldraw [black] (-1,-1) circle (1.5pt); 
	\filldraw [black] (-0.55,-0.73) circle (1.5pt); 
	\draw[ultra thick] (0,-0.64) -- (0.6,-0.7) -- (1.03,-0.95) -- (0.7,-1.25) -- (0.1,-1.35) -- (-0.55,-1.3) -- (-1,-1) --(-0.55,-0.73) -- (0,-0.64);
	\filldraw [black] (0,0) circle (1.5pt); 
	\draw[ultra thick] (0,0) -- (-0.55,-1.3);
	\draw[ultra thick] (0,0) -- (0.7,-1.25);
	\draw[ultra thick] (0,0) -- (-1,-1);
	\draw[ultra thick] (0,0) -- (0.6,-0.7);
	\coordinate [label=below:Pyramid]  (A) at (0,-1.7);
\end{tikzpicture}
\begin{tikzpicture}
	\coordinate [label=left:$\quad $]  (A) at (-1.7,0);
	\filldraw [black] (0,0) circle (1.5pt); 
	\filldraw [black] (-0.2,-0.5) circle (1.5pt); 
	\filldraw [black] (-0.3,-1) circle (1.5pt); 
	\filldraw [black] (0.8,-0.9) circle (1.5pt); 
	\filldraw [black] (0,-1.45) circle (1.5pt); 
	\filldraw [black] (-0.55,-1.3) circle (1.5pt); 
	\filldraw [black] (0.55,-1.3) circle (1.5pt); 
	\filldraw [black] (-0.8,-0.9) circle (1.5pt); 
	\filldraw [black] (-0.65,-0.5) circle (1.5pt); 
	\filldraw [black] (0.65,-0.5) circle (1.5pt); 
	\draw[ultra thick] (0,0) -- (0.65,-0.5) -- (0.8,-0.9) -- (0.55,-1.3) -- (0,-1.45) -- (-0.55,-1.3) -- (-0.8,-0.9) -- (-0.65,-0.5)-- (0,0) ;
	\draw[ultra thick] (0,0) -- (-0.2,-0.5) -- (-0.3,-1) -- (0,-1.45);
	\coordinate [label=below:Trihedron]  (A) at (0,-1.7);
\end{tikzpicture}
\begin{tikzpicture}
	\coordinate [label=left:$\quad $]  (A) at (-1.7,0);
	\filldraw [black] (0.4,0) circle (1.5pt); 
	\filldraw [black] (1.03,-0.65) circle (1.5pt); 
	\filldraw [black] (0,-0.5) circle (1.5pt); 
	\filldraw [black] (0.7,-1.25) circle (1.5pt);
	\filldraw [black] (0,-1) circle (1.5pt); 
	\filldraw [black] (0.1,-1.45) circle (1.5pt); 
	\filldraw [black] (0,-0.2) circle (1.5pt); 
	\filldraw [black] (0.1,-1.45) circle (1.5pt); 
	\filldraw [black] (-0.55,-1.3) circle (1.5pt); 
	\filldraw [black] (-0.8,-0.9) circle (1.5pt); 
	\filldraw [black] (-0.7,-0.5) circle (1.5pt); 
	\draw[ultra thick] (0.4,0) -- (1.03,-0.65) -- (0.7,-1.25) -- (0.1,-1.45) -- (-0.55,-1.3) -- (-0.8,-0.9) -- (-0.7,-0.5) -- (-0.2,0) -- (0.4,0);
	\filldraw [black] (-0.2,0) circle (1.5pt);
	\draw[ultra thick] (0,-0.2) -- (0.4,0);
	\draw[ultra thick] (0,-0.2) -- (-0.2,0);
	\draw[ultra thick] (0,-0.2) -- (0,-0.5) -- (0,-1) -- (0.1,-1.45);
	\coordinate [label=below:Stamp]  (A) at (0.1,-1.7);
\end{tikzpicture}
\begin{tikzpicture}
	\coordinate [label=left:$\quad $]  (A) at (-1.4,0);
	\filldraw [black] (0.7,0) circle (1.5pt); 
	\filldraw [black] (-0.3,0) circle (1.5pt); 
	\filldraw [black] (0,-0.3) circle (1.5pt); 
	\filldraw [black] (0.7,-1) circle (1.5pt); 
	\filldraw [black] (-0.3,-1) circle (1.5pt); 
	\filldraw [black] (0,-1.3) circle (1.5pt); 
	\filldraw [black] (-0.4,-0.4) circle (1.5pt); 
	\filldraw [black] (-0.4,-0.7) circle (1.5pt); 
	\filldraw [black] (0.8,-0.4) circle (1.5pt); 
	\filldraw [black] (0.8,-0.7) circle (1.5pt); 
	\filldraw [black] (0.1,-0.6) circle (1.5pt); 
	\filldraw [black] (-0.1,-0.9) circle (1.5pt); 
	\coordinate [label=below:Prism]  (A) at (0.2,-1.7);
	\draw[ultra thick] (0.7,0) -- (0.8,-0.4) -- (0.8,-0.7) -- (0.7,-1) -- (0,-1.3) -- (-0.3,-1) -- (-0.4, -0.7) -- (-0.4, -0.4) -- (-0.3,0) -- (0.7,0);
	\draw[ultra thick] (-0.3,0) -- (0,-0.3) -- (0.7,0);
	\draw[ultra thick] (0,-0.3) -- (0.1,-0.6) -- (-0.1,-0.9) -- (0,-1.3);
	\draw[ultra thick] (-0.3,-1) -- (0.7,-1);
\end{tikzpicture}
\end{equation*}
\begin{equation*}
	\text{Figure 2 : Solid graphs}
\end{equation*}
These four graphs are called \textit{solid graphs}, and the induced cycles in a solid graph are called the \textit{faces} of the solid graph. If $G=(V,E)$ is a graph, an induced subgraph of $G$ that is isomorphic to a solid graph is called a \textit{solid in $G$}. $S(G)$ denotes the set of all solids in $G$. 
\end{definition}

Solid graphs are compact, in the sense that any nonedge of them is lying on some induced cycle. This observation yields a nice criterion for a subgraph isomorphic to solid graph to be induced subgraph, and hence solid.  

\begin{proposition}
	Let $G=(V,E)$ be a graph with a subgraph $H$ which is isormorphic to some solid graph. Then $H$ is an induced subgraph of $G$ if and only if $C(H)\subseteq C(G)$. 
\end{proposition}

\begin{proof}
	One direction is trivial. Suppose $C(H)\subseteq C(G)$. Suppose $H$ has a nonedge $uv$ such that $uv\in E$. But observe that each nonedge $uv$ of a solid graph, there is an induced cycle, say $C$, in that solid graph containing the vertices $u$ and $v$. Hence $C\in C(H)$ uses both $u$ and $v$. But since $C\in C(G)$, $uv$ must be a nonedge in $G$, contrary to the assumption. Hence $H\in S(G)$. 
\end{proof}

Therefore, a subgraph of $G$ isomorphic to a solid graph is a solid in $G$ if each of its "faces" are induced cycles in $G$. 

We would like to think of the solids in a graph $G$ as the 3-cells in $G$. This may need some topological justification. First of all, the solid graphs are indeed graphs of some polyhedra, so defining any induced subgraph of $G$ isomorphic to some solid graphs is not too bad as a definition of 3-cells in graphs. Then why it should be induced subgraphs? Recall that we have defined the induced cycles as the 2-cells in the graphs. Then how would the 3-cells compatible to the already defined lower dimensional cells in the graphs look like? They must be 3-dimensional objectes whose surface consists of the 2-cells, the induced cycles. Hence we better look for a special set of induced cycles of a given graph $G$, with the property that those cycles \textit{enclose} a 3-dimensional region. This lead us to the notion of \textit{closedness} and to the definition of closed surface in graphs. 

\begin{definition}
For a set $F$ of cycles, we define the underlying graph of $H$ by $|H|:=\bigcup H\subseteq G$. A set $F$ of cycles is \textit{closed} if for each edge $e$ of $|F|$ there are exactly two induced cycles $C_{1},C_{2}\in F$ that use the edge $e$. If $G=(V,E)$ is a graph, then a set of cycles $F$ is \textit{closed in $C(G)$} if $F\subseteq C(G)$ and it is closed. A subgraph $S$ of $G=(V,E)$ is called a \textit{closed surface in $G$} if $S=|F|$ for some closed subset $F$ of $C(G)$. In that case, we call $F$ a \textit{face set of $S$}.
\end{definition}

However, one might notice that the set of all closed surfaces in a graph $G$ seems too large than it ought to be. It contains not only the usuall polyhedrons, but also vertex sum, or even disjoint union of them. 

\begin{proposition}
	Let $G=(V,E)$ be a graph and let $F$ be a closed subset of $C(G)$. If $F'\subseteq F$ is closed, then $F\setminus F'$ is closed and the two closed surfaces $|F\setminus F'|$ and $|F'|$ does not have any common edge in $G$. Moreover, $F$ can be written as a finite union of closed sets $H_{1},\cdots,H_{k}$ such that 
\begin{description}
	\item{i)} for each $1\le i \le k$, $H_{i}$ is nonempty and $H_{i}$ has no nonempty proper subset which is closed
	\item{ii)} for each $1\le i,j\le k$, $|H|_{i}$ and $|H_{j}|$ have no common edge in $G$. 
\end{description}
In this case, we write $F=H_{1}\lor H_{2}\lor \cdots \lor H_{k}$ and such expression is called \textit{a decomposition of $F$}, and it is unique up to order of $H_{i}$s. 
\end{proposition}

\begin{proof}
	The existence of decomposition of $F$ follows from the first assertion. Let $F'$ be a subset of $F$ that is closed. If $C\in F\setminus F'$, then $C$ cannot use any edge of $E(|F'|)$ since each of them are already used by two cycles in $F'$. Thus $|F\setminus F'|$ and $|F'|$ have no common edge in $G$. Then since $F$ is closed, any edge $e$ of $|F\setminus F'|$ must be used by exactly two cycles of $F\setminus F'$, and this shows $F\setminus F'$ is closed.

For the uniqueness, we use induction on the size of the closed set $F$. Let $F=H_{1}\lor \cdots \lor H_{m}=H'_{1}\lor \cdots \lor H'_{n}$ be two decompositions of the closed set $F$. I claim that $H_{1}$ must appear on the second decomposition of $F$. To see this, denote $R_{i}:=H_{1}\cap H_{i}'$ so that $H_{1}$ is a disjoint union of $R_{1},\cdots,R_{k}$. We may suppose two of these are nonempty, say $R_{1}$ and $R_{2}$, for otherwise $H_{1}=H_{j}'$ for some $1\le j\le k$. If $H_{1}=R_{i}$ for some $i=1,2$, then $H_{1}=R_{i}=H_{i}'$ since $H_{1}$ does not contain a nonempty proper closed set. Hence we may suppose $R_{i}\subsetneq H_{i}'$. Then since $H_{1}'$ contains no nonempty proper closed subset, there is an edge $e$ of $|R_{1}|$ that is not used by two cycles of $R_{1}$. But since $H_{1}$ is closed, $e$ is used in two distinct cycles, say $C_{1},C_{2}$ of $H_{1}$. But then $R_{1}$ contains both $C_{1},C_{2}$, since distinct $|R_{i}|$ and $|R_{j}|$ have no common edge in $G$ by the first assertion. Therefore $e$ is used in two distinct cycles of $R_{1}$, which is a contradiction. This shows the claim. 

Hence we may suppose $H_{1}=H_{1}'$. Then $F\setminus H_{1}=H_{2}\lor \cdots \lor H_{m}=H_{2}'\lor \cdots \lor H_{m}'=F\setminus H_{1}'$ has unique decomposition by induction hypothesis, and hence those two decompositions of $F$ must be identical. 
\end{proof}

We want a closed surface which cannot be decomposed further. So we consider following minimality on closed surfaces:
\begin{definition}
A closed surface $S$ in a graph $G$ is \textit{minimal} if there is no closed surface $S'$ in $G$ such that $V(S')\subsetneq V(S)$, and is $\textit{strongly minimal}$ if for every closed surface $S'$ in $G$ such that $V(S')\subseteq V(S)$ we have $S'=S$. A subgraph $S\subseteq G$ is a  \textit{minimal closed surface} in $G$ if it is a closed surface in $G$ and it is minimal. 
\end{definition}

\begin{remark}
	A closed surface $S$ in $G$ is minimal if it is strongly minimal.
\end{remark}

As one might expect from the figures of solid graphs, the solids in a graph $G$ are minimal closed surfaces in $G$. At first, notice that each faces of a solid $S$ in $G$ is an induced cycle. Hence the set of all faces forms a closed subset of $C(G)$, and hence $S$ is a closed surface of $G$. Then what remains is to show the minimality. Some insightful reader might say that is trivial, but proving it in regorous way could be techenical. But fortunately, this can be done simply by a neat observation. Recall the fact that for any solid $S$ in $G$, $C(S)$ is a closed subset of $C(G)$ and $|C(S)|=S$. That is, the induced cycles in $S$ are just enough to form a closed surface, and such closed surface is $S$ itself. Note that not all closed surface $S$ satisfies this property; some closed surface may have additional induced cycles, such as a triangularization of the torus. Next proposition relates this property with the minimality of a closed surface.  

\begin{proposition}
	An induced closed surface $S$ of a graph $G=(V,E)$ is strongly minimal if $C(S)$ is closed and it has no proper decomposition. 
\end{proposition}

\begin{proof}
	Let $F$ be a closed subset of $C(G)$ with face set $F$. A cycle of $F$ is an induced cycle of $G$ contained in $S$, and hence $F\subseteq C(S)$. Then since $C(S)$ has no proper decomposition, we have $F=C(S)$ and thus $S=|F|=|C(S)|$. That is, $C(S)$ is a face set of $S$. 
	Now let $S'$ be a closed surface in $G$ with face set $F'$, such that $V(S')\subseteq V(S)$. Then the cycles of $F'$ are contained in $S$, and clearly an induced cycle in $G$ that is contained in $S$ is an induced cycle in $S$; thus $F'\subseteq C(S)$. But since $C(S)$ does not have a proper decomposition, we then have $F'=C(S)$ and hence $S'=|F|=|C(S)|=S$. Therefore $S$ is strongly minimal. 
\end{proof}

Now it is easy to show that solid graphs have the strong minimality.

\begin{proposition}
	Let $G=(V,E)$ be a graph. A solid $S$ in $G$ is a strongly minimal closed surface. In particular, we have $S(G)\subseteq \mathcal{S}(G)$.
\end{proposition}

\begin{proof}
	It is easy to see from the graphs of the solid graphs that $C(S)$ is closed and have no proper decomposition. Hence solids in $G$ are stongly minimal surfaces. 
\end{proof}

One could guess that the minimal closed surfaces in a graph would be the boundaries of homeomorph of $S^{2}$, the 2-dimensional sphere. The solid graphs are indeed so. Now that we know four kinds of minimal closed surface, asking whether there are more of them is quite natural. \textit{In fact, we have found all} ; the two sets $S(G)$ and $\mathcal{S}(G)$ actually coincides. Proposition 4.1.1 shows one inclusion, and proving the reverse inclusion, i.e., the characterization of minimal closed surfaces, is postponed to the appendix. But note that only the definition and minimality of solids will be used in developing the theory of the third graph characteristic.

We would like to define the solids in a graph $G$ as the 3-cells in $G$. This may need some topological justification. First of all, the solid graphs are indeed graphs of some polyhedrons, so defining any induced subgraph of $G$ isomorphic to some solid graphs is not too bad as a definition of 3-cells in graphs. Then why it should be induced subgraphs? Recall that we have defined the induced cycles as the 2-cells in the graphs. Then how would the 3-cells compatible to the already defined lower dimensional cells in the graphs look like? They must be 3-dimensional objectes whose surface consists of the 2-cells, the induced cycles. Hence we better look for a special set of induced cycles of a given graph $G$, with the property that those cycles \textit{enclose} a 3-dimensional region. This lead us to the notion of \textit{closedness} and to the definition of closed surface in graphs. 

\begin{definition}
For a set $F$ of cycles, we define the underlying graph of $H$ by $|H|:=\bigcup H\subseteq G$. A set $F$ of cycles is \textit{closed} if for each edge $e$ of $|F|$ there are exactly two induced cycles $C_{1},C_{2}\in F$ that use the edge $e$. If $G=(V,E)$ is a graph, then a set of cycles $F$ is \textit{closed in $C(G)$} if $F\subseteq C(G)$ and it is closed. A subgraph $S$ of $G=(V,E)$ is called a \textit{closed surface in $G$} if $S=|F|$ for some closed subset $F$ of $C(G)$. In that case, we call $F$ a \textit{face set of $S$}.
\end{definition}

However, one might notice that the set of all closed surfaces in a graph $G$ seems too large than it ought to be. It contains not only the usuall polyhedrons, but also vertex sum, or even disjoint union of them. 

\begin{proposition}
	Let $G=(V,E)$ be a graph and let $F$ be a closed subset of $C(G)$. If $F'\subseteq F$ is closed, then $F\setminus F'$ is closed and the two closed surfaces $|F\setminus F'|$ and $|F'|$ does not have any common edge in $G$. Moreover, $F$ can be written as a finite union of closed sets $H_{1},\cdots,H_{k}$ such that 
\begin{description}
	\item{i)} for each $1\le i \le k$, $H_{i}$ is nonempty and $H_{i}$ has no nonempty proper subset which is closed
	\item{ii)} for each $1\le i,j\le k$, $|H|_{i}$ and $|H_{j}|$ have no common edge in $G$. 
\end{description}
In this case, we write $F=H_{1}\lor H_{2}\lor \cdots \lor H_{k}$ and such expression is called \textit{a decomposition of $F$}, and it is unique up to order of $H_{i}$s. 
\end{proposition}

\begin{proof}
	The existence of decomposition of $F$ follows from the first assertion. Let $F'$ be a subset of $F$ that is closed. If $C\in F\setminus F'$, then $C$ cannot use any edge of $E(|F'|)$ since each of them are already used by two cycles in $F'$. Thus $|F\setminus F'|$ and $|F'|$ have no common edge in $G$. Then since $F$ is closed, any edge $e$ of $|F\setminus F'|$ must be used by exactly two cycles of $F\setminus F'$, and this shows $F\setminus F'$ is closed.

For the uniqueness, we use induction on the size of the closed set $F$. Let $F=H_{1}\lor \cdots \lor H_{m}=H'_{1}\lor \cdots \lor H'_{n}$ be two decompositions of the closed set $F$. I claim that $H_{1}$ must appear on the second decomposition of $F$. To see this, denote $R_{i}:=H_{1}\cap H_{i}'$ so that $H_{1}$ is a disjoint union of $R_{1},\cdots,R_{k}$. We may suppose two of these are nonempty, say $R_{1}$ and $R_{2}$, for otherwise $H_{1}=H_{j}'$ for some $1\le j\le k$. If $H_{1}=R_{i}$ for some $i=1,2$, then $H_{1}=R_{i}=H_{i}'$ since $H_{1}$ does not contain a nonempty proper closed set. Hence we may suppose $R_{i}\subsetneq H_{i}'$. Then since $H_{1}'$ contains no nonempty proper closed subset, there is an edge $e$ of $|R_{1}|$ that is not used by two cycles of $R_{1}$. But since $H_{1}$ is closed, $e$ is used in two distinct cycles, say $C_{1},C_{2}$ of $H_{1}$. But then $R_{1}$ contains both $C_{1},C_{2}$, since distinct $|R_{i}|$ and $|R_{j}|$ have no common edge in $G$ by the first assertion. Therefore $e$ is used in two distinct cycles of $R_{1}$, which is a contradiction. This shows the claim. 

Hence we may suppose $H_{1}=H_{1}'$. Then $F\setminus H_{1}=H_{2}\lor \cdots \lor H_{m}=H_{2}'\lor \cdots \lor H_{m}'=F\setminus H_{1}'$ has unique decomposition by induction hypothesis, and hence those two decompositions of $F$ must be identical. 
\end{proof}

We want a closed surface which cannot be decomposed further. So we consider following minimality on closed surfaces:
\begin{definition}
A closed surface $S$ in a graph $G$ is \textit{minimal} if there is no closed surface $S'$ in $G$ such that $V(S')\subsetneq V(S)$, and is $\textit{strongly minimal}$ if for every closed surface $S'$ in $G$ such that $V(S')\subseteq V(S)$ we have $S'=S$. A subgraph $S\subseteq G$ is a  \textit{minimal closed surface} in $G$ if it is a closed surface in $G$ and it is minimal. 
\end{definition}

\begin{remark}
	A closed surface $S$ in $G$ is minimal if it is strongly minimal.
\end{remark}

As one might expect from the figures of solid graphs, the solids in a graph $G$ are minimal closed surfaces in $G$. At first, notice that each faces of a solid $S$ in $G$ is an induced cycle. Hence the set of all faces forms a closed subset of $C(G)$, and hence $S$ is a closed surface of $G$. Then what remains is to show the minimality. Some insightful reader might say that is trivial, but proving it in regorous way could be techenical. But fortunately, this can be done simply by a neat observation. Recall the fact that for any solid $S$ in $G$, $C(S)$ is a closed subset of $C(G)$ and $|C(S)|=S$. That is, the induced cycles in $S$ are just enough to form a closed surface, and such closed surface is $S$ itself. Note that not all closed surface $S$ satisfies this property; some closed surface may have additional induced cycles, such as a triangularization of the torus. Next proposition relates this property with the minimality of a closed surface.  

\begin{proposition}
	An induced closed surface $S$ of a graph $G=(V,E)$ is strongly minimal if $C(S)$ is closed and it has no proper decomposition. 
\end{proposition}

\begin{proof}
	Let $F$ be a closed subset of $C(G)$ with face set $F$. A cycle of $F$ is an induced cycle of $G$ contained in $S$, and hence $F\subseteq C(S)$. Then since $C(S)$ has no proper decomposition, we have $F=C(S)$ and thus $S=|F|=|C(S)|$. That is, $C(S)$ is a face set of $S$. 
	Now let $S'$ be a closed surface in $G$ with face set $F'$, such that $V(S')\subseteq V(S)$. Then the cycles of $F'$ are contained in $S$, and clearly an induced cycle in $G$ that is contained in $S$ is an induced cycle in $S$; thus $F'\subseteq C(S)$. But since $C(S)$ does not have a proper decomposition, we then have $F'=C(S)$ and hence $S'=|F|=|C(S)|=S$. Therefore $S$ is strongly minimal. 
\end{proof}

Now it is easy to show that solid graphs have the strong minimality.

\begin{proposition}
	Let $G=(V,E)$ be a graph. A solid $S$ in $G$ is a strongly minimal closed surface. In particular, we have $S(G)\subseteq \mathcal{S}(G)$.
\end{proposition}

\begin{proof}
	It is easy to see from the graphs of the solid graphs that $C(S)$ is closed and have no proper decomposition. Hence solids in $G$ are stongly minimal surfaces. 
\end{proof}

One could guess that the minimal closed surfaces in a graph would be the boundaries of homeomorph of $S^{2}$, the 2-dimensional sphere. The solid graphs are indeed so. Now that we know four kinds of minimal closed surface, asking whether there are more of them is quite natural. \textit{In fact, we have found all} ; the two sets $S(G)$ and $\mathcal{S}(G)$ actually coincides. Proposition 4.1.1 shows one inclusion, and proving the reverse inclusion requires several steps. Only the definition and minimality of solids will be used in developing the theory of the third graph characteristic. Below are the proof of the reverse inclusion, which is also a characterization of minimal closed surfaces in a graph.  

\begin{proposition}
	Let $C$ be a cycle in a graph $G$. Then there is a subset $R$ of $C(G)$ with following properties: 
\begin{description}
	\item{i) $V(C) = V(|R|)$}
	\item{ii) for each edge $e$ in $C$, there is exactly one cycle in $R$ using $e$}
	\item{iii) for each edge $e\in E(|R|)\setminus E(C)$, there are exactly two cycles in $R$ using $e$	}

\end{description}
Such subset $R\in C(G)$ is called a \textit{refinement of $C$}.
\end{proposition}

\begin{proof}
	Induction on $|V(C)|$. If $C$ is an induced cycle in $G$, we are done. Otherwise, there is an edge $e$ between two vertices in $C$ such that $C+e$ is an edge sum of an induced cycle $C_{1}$ and a cycle $C'$ with $E(C_{1})\cap E(C')=\{e\}$. Apply induction hypothesis to $C'$ to obtain a refinement $R'$ of it. Then $R:=\{C_{1}\}\cup R'$ is an refinement of $C$.
\end{proof}

\begin{proposition}
	Let $G$ be a graph with an induced cycle $C$. Let $x,y$ be two nonadjacent vertices in $C$. Suppose $P$ is an induced path from $x$ to $y$, which does not use any of the vertices in $C$ except $x,y$. Then there is closed subset $F$ of $C(G)$ such that $V(|F|)=V(C\cup P)$ and $C\in F$.
\end{proposition}

\begin{proof}
	Let $P_{1}, P_{2}$ be two induced paths from $x$ to $y$ in $C$. Then $C':=P_{1}\cup P$ is a cycle, and any possible crossing edges in $C'$ lies between $P_{1}$ and $P$; for, observe that $xy\notin E(G)$ since $xy\notin E(C)$ and $C\in C(G)$. Also note that $P_{1}$ and $P$ are induced paths in $G$. Now apply Proposition 13 to get a refinement $R'$ of $C'$. Then by previous observation, any induced cycle in $R'$ must use at least one vertex in $P_{1}$. Hence each of the induced cycles in $\{C\}\sqcup S'$ uses at least one vertex in $P_{1}$. Observe that each edges in $C\cup |F'|$ is used exactly twice by the induced cycles in $\{C\}\cup R'$, except the edges in $P_{2}\cup P$; they are used only once. On the other hand, let $R''$ be a refinement of the cycle $P_{2} \cup P$. Each edge in $|R|''$ is used twice, except the ones in $P_{2}\cup P$. Note that no induced cycles in $R''$ uses any vertices in $P_{1}$. Hence $R''$ and $\{C\}\sqcup R'$ are disjoint. Thus $F:=\{C\}\sqcup R' \sqcup R''$ is a closed subset of $C(G)$. 
\end{proof}

\begin{proposition}
	Let $S$ be a minimal closed surface in a graph $G$. For each vertex $v$ in $S$, either of the two holds:
\begin{description}
	\item{i)} there are two faces $C_{1},C_{2}$ of $S$ using $v$ such that $C_{1}\cup C_{2}$ is a spanning subgraph of $S$ isomorphic to trihedron graph, and $\text{deg}_{S}(v) \le 3$

	\item{ii)} $S$ contains a cycle $C$ not using $v$ such that $V(S)=V(C+v)$ and $\text{deg}_{C+v}(v)\ge 3$.
\end{description}
\end{proposition}

\begin{proof}
	Let $x_{1},\cdots,x_{r}$ be the neighbors of $v$ in $S$. Suppose that there are two faces $C_{1},C_{2}$ of $S$ both using an edge $vx_{j}$ for some $1\le j \le r$, such that $|V(C_{1}\cap C_{2})|\ge 3$. Consider $C_{2}$ as an induced path from $v$ to $v$ along $C_{2}$, fixing one direction. Let $v=y_{1},\cdots, y_{l}=x_{j}$($l\ge 3$) be the vertices in $C_{1}$ at which $C_{2}$ intersects with $C_{1}$ in order. Denote the induced path from $y_{i}$ to $y_{i+1}$ along $C_{2}$ by $P_{i}$. Observe that each $P_{i}$ is of length 1 if and only if its two endpoints $y_{i}$ and $y_{i+1}$ are adjacent in $C_{1}$, because $C_{1},C_{2}$ are induced cycles of $G$. Now suppose some path $P_{k}$ is of length $\ge 2$. Then Proposition 4.1.7 says that there is a closed surface $S'$ in $G$ such that $V(S')=V(C_{1}\cup P_{i})$. Then since
\begin{equation*}
	V(S')=V(C_{1}\cup P_{k}) \subseteq V(C_{1}\cup C_{2})\subseteq V(S),
\end{equation*}
the minimality of $S$ yields $V(C_{1}\cup P_{k})=V(C_{1}\cup C_{2})=V(S)$ and consequently $k$ is the unique index for which $P_{k}$ is of length $\ge 2$. This shows $C_{1}\cup P_{k}=C_{1}\cup C_{2}$ is isomorphic to a trihedron graph, with $y_{k}$ and $y_{k+1}$ the two degree three vertices in it. Moreover, notice that $\Gamma_{S}(v)=\Gamma_{C_{1}}(v)\cup \Gamma_{C_{2}}(v)$ and $x_{j}\in \Gamma_{C_{1}}(v)\cup \Gamma_{C_{2}}(v)$. This shows $\text{deg}_{S}(v)\le 3$. 

On the other hand, suppose there are no such two cycles using $v$. Note that each face of $S$ using $v$ uses exactly two neighbors of $v$. Thus $v$ has at least three neighbors and hence at least three faces of $S$ using it. Therefore if $C_{1},\cdots,C_{r}$ are the faces of $S$ using $v$, then $C:=\bigcup_{i=1}^{r}C_{i}-v$ is a cycle in $G$ which does not use $v$ and $\text{deg}_{C+v}(v)\ge 3$. Now let $R$ be a refinement of $C$. Then $F:=R\cup \{C_{1},\cdots,C_{r}\}$ forms a closed subset of $C(G)$ and $V(C+v)=V(|F|)\subseteq V(S)$. Hence by minimality of $S$ we have $V(C+v)=V(S)$ as required. 
\end{proof}

Next Lemma does most of the work for classification of minimal closed surfaces. 

\begin{lemma}
	Let $G$ be a graph with a minimal closed surface $S$. Suppose there are cycle $C$ and vertex $w$ in $G$ such that $w\notin V(C)$, and $C+w$ is a subgraph of $S$. If $\deg_{C+w}(w)\ge 4$, then $S=C+w$ and it is a pyramid with apex $w$, and if $\deg_{C+w}(w)=3$, then $V(S)=V(C+w)$ and $S$ is either a pyramid or a trihedron in $G$. 
\end{lemma}

\begin{proof}
	Let $R$ be a refinement of $C$ and let $H$ be the set of induced cycles in $C+w$ that uses $w$. Then $R\cup H$ forms a closed subset of $C(G)$, and hence we have $V(S)\subseteq V(C+w)$ by the minimality of $S$. Thus we have $V(S)=V(C+w)$. 

Suppose the cycle $C$ has a crossing edge $e$ in $G$, that is, $C$ is not an induced cycle of $G$. The edge $e$ divides $C$ into two subcycles $C_{1}$ and $C_{2}$. Then I claim that if one of them contains at least three neighbors of $w$, say $C_{1}$, then by the first paragraph of the proof we have $V(S)=V(C_{1}+w)$, which is a contradiction since $V(S)=V(C+w)\supsetneq V(C_{1}+w)$, proving the claim. Now it is immediate that if $w$ has at least five neighbors on $C$ then $C\in C(G)$, since any crossing edge $e$ divides $C$ into two subcycles and at least one of them must contain at least three out of the $\ge 5$ neighbors of $w$. 

Suppose $C$ has four neighbors of $w$. Let $C_{1},C_{2},C_{3}$ and $C_{4}$ be the four induced cycles in $C+w$ that uses $w$, and let $x_{1},x_{2},x_{3},x_{4}$ be the four neighbors of $w$ as described in Figure 3.1. Then by the claim in the second paragraph, we can assume that any crossing edge lies between the interior points of the paths $C_{1}-w$ and $C_{3}-w$ or $C_{2}-w$ and $C_{4}-w$. Without loss of generality, suppose there is a crossing edge $e=uv$ between interior points of the the paths $C_{1}-w$ and $C_{3}-w$. Let $P_{1},P_{2}$ be the paths from $u$ to $x_{2}$ and $v$ to $x_{4}$ along the paths $C_{1}-w$ and $C_{3}-w$, repectively.(possibly $P_{1}=x_{4}$ or $P_{2}=x_{2}$) We may assume that $e$ is such that the cycle $C':=P_{1}\cup P_{2}+ e+w $ is the smallest one. Then $C'$ is an induced cycle in $G$ by the choice and the assumption that no crossing edge of $C$ is indicent to any of the four neighbors of $w$ on $C$. Now let $C''$ be one of the two sub-cycles of $C$ that uses the edge $e$, containing $C_{4}-w$. Let $R''$ be a refinement of $C''$. Then $H'':=R''\cup C_{1}\cup C' \cup C_{4}$ is a closed subset of $C(G)$ such that $|H''|$ does not contain the vertex $x_{3}$.(see figure 3.2) This contradicts the minimality of $S$, and hence $C\in C(G)$ if $w$ has $\ge 4$ neighbors on $C$. This shows the first assertion. 
\begin{equation*}
\begin{tikzpicture}
	\filldraw [black] (0,0) circle (1.5pt); \coordinate [label=left:$w$]  (A) at (0,0);
	\draw[thick] (0,0) circle (25pt);
	\filldraw [black] (0.63,0.63) circle (1.5pt); \coordinate [label=right:$x_{1}$]  (A) at (0.7,0.7);
	\filldraw [black] (0.63,-0.63) circle (1.5pt); \coordinate [label=right:$x_{2}$]  (A) at (0.7,-0.7);
	\filldraw [black] (-0.63,-0.63) circle (1.5pt); \coordinate [label=left:$x_{3}$]  (A) at (-0.7,-0.7);
	\filldraw [black] (-0.63,0.63) circle (1.5pt); \coordinate [label=left:$x_{4}$]  (A) at (-0.7,0.7);
	\draw[thick] (0.63,0.63) -- (-0.63,-0.63); 	
	\draw[thick] (0.63,-0.63) -- (-0.63,0.63);
	\coordinate [label=right:$C_{1}$]  (A) at (0.25,0);
	\coordinate [label=below:$C_{2}$]  (A) at (0,-0.25);
	\coordinate [label=left:$C_{3}$]  (A) at (-0.25,0);
	\coordinate [label=above:$C_{4}$]  (A) at (0,0.25);
	\coordinate [label=below:Figure 3.1]  (A) at (0,-1.2);
\end{tikzpicture}	
\begin{tikzpicture}
	\coordinate [label=left:$\qquad$]  (A) at (-2,0);
	\filldraw [black] (0,0) circle (1.5pt); \coordinate [label=left:$w$]  (A) at (0,0);
	\draw[thick] (0,0) circle (25pt);
	\filldraw [black] (0.63,0.63) circle (1.5pt); \coordinate [label=right:$x_{1}$]  (A) at (0.7,0.7);
	\filldraw [black] (0.63,-0.63) circle (1.5pt); \coordinate [label=right:$x_{2}$]  (A) at (0.7,-0.7);
	\filldraw [black] (-0.63,-0.63) circle (1.5pt); \coordinate [label=left:$x_{3}$]  (A) at (-0.7,-0.7);
	\filldraw [black] (-0.63,0.63) circle (1.5pt); \coordinate [label=left:$x_{4}$]  (A) at (-0.7,0.7);
	\filldraw [black] (-0.63,0.63) circle (1.5pt); \coordinate [label=left:$x_{4}$]  (A) at (-0.7,0.7);
	\filldraw [black] (0.83,-0.3) circle (1.5pt);  \coordinate [label=right:$u$]  (A) at (0.83,-0.3);
	\filldraw [black] (-0.86,-0.1) circle (1.5pt); \coordinate [label=left:$v$]  (A) at (-0.86,-0.1);
	\draw[thick] (0.63,0.63) -- (-0.63,-0.63); 
	\draw[ultra thick] (0.63,-0.63) -- (-0.63,0.63);
	\draw[ultra thick] (0.83,-0.3) -- (-0.86,-0.1);
	\draw[ultra thick] (0.63,-0.63) arc (-45:-16:25pt);
	\draw[ultra thick] (-0.63,0.63) arc (135:185:25pt);
	\coordinate [label=right:$C_{1}$]  (A) at (0.25,0);
	\coordinate [label=above:$C_{4}$]  (A) at (0,0.25);
	\coordinate [label=right:$C'$]  (A) at (-0.9,0.2);
	\coordinate [label=below:Figure 3.2]  (A) at (0,-1.2);
\end{tikzpicture}	
\end{equation*}

For the second assertion, suppose $C$ has three neighbors of $w$. Let $C_{1},C_{2},C_{3}$ and $x_{1},x_{2},x_{3}$ as in figure 3. By the claim in the second paragraph, we may assume that no crossing edge of $C$ lie in some cycle $C_{1},C_{2}$ or $C_{3}$. Suppose there is a crossing edge of $C$ that uses one of $x_{1},x_{2}$, and $x_{3}$; let $e=vx_{1}$ be such edge using $x_{1}$, where $v\in V(C_{2})$. Assume $C_{3}-w$ is a path of length $\ge 2$, so that there is an interior point, say $u$, of the path $C_{3}-w$. Then we may choose $e$ among such crossing edges for which the bold cycle $C'$ in Figure 4 is the smallest; similarly, then $C'$ is an induced cycle of $G$. Let $C''$ be the one of the two subcycles of $C$ given by the crossing edge $e$ that use the vertex $x_{2}$, and let $R$ be a refinement of it. Then $F:=\{C_{1},C_{2},C'\}\cup R$ forms a closed subset of $C(G)$ such that $V(|F|)\subseteq V(S)\setminus \{u\}$, which contradicts the minimality of $S$. Thus we may assume both the paths $C_{3}-w$ and $C_{1}-w$ are of length $1$, i.e., they are just edges. Then notice that the vertex $x_{1}$ has four neighbors $x_{3},v,x_{2}$, and $w$ on $C_{2}$ (and possibly more), so that by the first assertion we have $S':=C_{2}+x_{1}$ is a pyramid and note that $V(S')=V(S)$. Then by the strong minimality, we have $S=S'$, and hence $S$ is a pyramid. 

Lastly, suppose there is a crossing edge of $C$ between the interior points of the paths $C_{1}-w$ and $C_{2}-w$. Suppose $e=uv$ is one of such crossing edge such that one the two paths from $v$ to $x_{2}$ along $C_{2}-w$ and the one from $u$ to $x_{2}$ is of length $\ge 2$. (See Figure 3.6) We may suppose the former has an interior point $z$. We can assume the crossing edge $e$ is the one for which the bold cycle $C'$ in Figure 3.6 is the smallest. Then it should be an induced cycle, similar argument for Figure 3.4 shows that we can get a closed surface contained in $S$ using less vertices. This contradicts the minimality, and hence we may assume all such crossing edge $e$ is that $u,v$ are adjacent to $x_{2}$ in $C$. 
\begin{equation*}
\begin{tikzpicture}
	\filldraw [black] (0,0) circle (1.5pt); \coordinate [label=left:$w$]  (A) at (0,0);
	\draw[thick] (0,0) circle (25pt);
	\filldraw [black] (0.63,0.63) circle (1.5pt); \coordinate [label=right:$x_{1}$]  (A) at (0.7,0.7);
	\filldraw [black] (-0.63,0.63) circle (1.5pt); \coordinate [label=left:$x_{3}$]  (A) at (-0.7,0.7);
	\filldraw [black] (0,-0.87) circle (1.5pt); \coordinate [label=below:$x_{2}$]  (A) at (0,-0.9);
	\draw[thick] (0.63,0.63) -- (0,0); 	
	\draw[thick] (-0.63,0.63) -- (0,0);
	\draw[thick] (0,0) -- (0,-0.87);
	\coordinate [label=right:$C_{1}$]  (A) at (0.2,-0.2);
	\coordinate [label=left:$C_{2}$]  (A) at (-0.2,-0.2);
	\coordinate [label=above:$C_{3}$]  (A) at (0,0.25);
	\coordinate [label=below: Figure 3.3]  (A) at (0,-1.4);
\end{tikzpicture}	
\begin{tikzpicture}
	\coordinate [label=left:$\qquad$]  (A) at (-2,0);
	\filldraw [black] (0,0) circle (1.5pt); \coordinate [label=left:$w$]  (A) at (0,-0.2);
	\draw[thick] (0,0) circle (25pt);
	\filldraw [black] (0.63,0.63) circle (1.5pt); \coordinate [label=right:$x_{1}$]  (A) at (0.7,0.7);
	\filldraw [black] (-0.63,0.63) circle (1.5pt); \coordinate [label=left:$x_{3}$]  (A) at (-0.7,0.7);
	\filldraw [black] (0,-0.87) circle (1.5pt); \coordinate [label=below:$x_{2}$]  (A) at (0,-0.9);
	\filldraw [black] (0,0.87) circle (1.5pt); \coordinate [label=above:$u$]  (A) at (0,0.87);
	\filldraw [black] (-0.86,-0.1) circle (1.5pt); \coordinate [label=left:$v$]  (A) at (-0.86,-0.1);
	\draw[ultra thick] (0.63,0.63) -- (0,0); 	
	\draw[ultra thick] (-0.63,0.63) -- (0,0);
	\draw[ultra thick] (-0.86,-0.1) -- (0.63,0.63);
	\draw[thick] (0,0) -- (0,-0.87);
	\coordinate [label=above:$C_{3}$]  (A) at (0,0.25);
	\coordinate [label=below:Figure 3.4]  (A) at (0,-1.4);
	\coordinate [label=right:$C_{1}$]  (A) at (0.2,-0.2);
	\coordinate [label=right:$C'$]  (A) at (-0.9,0.2);
	\coordinate [label=left:$C_{2}$]  (A) at (-0.1,-0.5);
	\draw[ultra thick]  (-0.86,-0.1) arc (188:135:25pt);
\end{tikzpicture}	
\begin{tikzpicture}
	\coordinate [label=left:$\qquad$]  (A) at (-2,0);
	\filldraw [black] (0,0) circle (1.5pt); \coordinate [label=left:$w$]  (A) at (0,-0.2);
	\draw[thick] (0,0) circle (25pt);
	\filldraw [black] (0.63,0.63) circle (1.5pt); \coordinate [label=right:$x_{1}$]  (A) at (0.7,0.7);
	\filldraw [black] (-0.63,0.63) circle (1.5pt); \coordinate [label=left:$x_{3}$]  (A) at (-0.7,0.7);
	\filldraw [black] (0,-0.87) circle (1.5pt); \coordinate [label=below:$x_{2}$]  (A) at (0,-0.9);
	\filldraw [black] (-0.86,-0.1) circle (1.5pt); \coordinate [label=left:$v$]  (A) at (-0.86,-0.1);
	\draw[thick] (0.63,0.63) -- (0,0); 	
	\draw[thick] (-0.63,0.63) -- (0,0);
	\draw[thick] (-0.86,-0.1) -- (0.63,0.63);
	\draw[thick] (0,0) -- (0,-0.87);
	\coordinate [label=above:$C_{3}$]  (A) at (0,0.25);
	\coordinate [label=below:Figure 3.5]  (A) at (0,-1.4);
	\coordinate [label=right:$C_{1}$]  (A) at (0.2,-0.2);
	\coordinate [label=left:$C_{2}$]  (A) at (-0.1,-0.5);
\end{tikzpicture}	
\begin{tikzpicture}
	\coordinate [label=left:$\qquad$]  (A) at (-2,0);
	\filldraw [black] (0,0) circle (1.5pt); \coordinate [label=left:$w$]  (A) at (0,0);
	\draw[thick] (0,0) circle (25pt);
	\filldraw [black] (0.63,0.63) circle (1.5pt); \coordinate [label=right:$x_{1}$]  (A) at (0.7,0.7);
	\filldraw [black] (-0.63,0.63) circle (1.5pt); \coordinate [label=left:$x_{3}$]  (A) at (-0.7,0.7);
	\filldraw [black] (0,-0.87) circle (1.5pt); \coordinate [label=below:$x_{2}$]  (A) at (0,-0.9);
	\filldraw [black] (0.65,-0.6) circle (1.5pt);  \coordinate [label=right:$u$]  (A) at (0.65,-0.6);
	\filldraw [black] (-0.65,-0.6) circle (1.5pt); \coordinate [label=left:$z$]  (A) at (-0.65,-0.6);
	\filldraw [black] (-0.86,-0.1) circle (1.5pt); \coordinate [label=left:$v$]  (A) at (-0.86,-0.1);
	\draw[thick] (0.63,0.63) -- (0,0); 	
	\draw[ultra thick] (-0.63,0.63) -- (0,0);
	\draw[ultra thick] (-0.86,-0.1) -- (0.65,-0.6);
	\draw[ultra thick] (0,0) -- (0,-0.87);
	\coordinate [label=above:$C_{3}$]  (A) at (0,0.25);
	\coordinate [label=below:Figure 3.6]  (A) at (0,-1.4);
	\coordinate [label=right:$C_{1}$]  (A) at (0.2,-0.2);
	\coordinate [label=right:$C'$]  (A) at (-0.9,0.2);
	\draw[ultra thick] (0,-0.87) arc (-90:-42:25pt);
	\draw[ultra thick]  (-0.86,-0.1) arc (188:135:25pt);
\end{tikzpicture}	
\end{equation*}

At this point we have narrowed down the situation into the one in Figure 3.7, where the bottom bold cycle is a triangle. Then the path $C_{3}-w$ from $x_{1}$ to $x_{3}$ must be an edge, since otherwise $C+w-x_{2}$ becomes a closed surface by Proposition 4.1.5.(set $C=C_{3}$, and $P$ the path from $x_{3}$ to $x_{1}$ passing through $u$ and $v$ in the Figure) Now it becomes a trihedron if the cycle, which is the  subcycle of $C$ given by $uv$ not using $x_{2}$, is induced. Denote this cycle by $C_{4}$ and suppose there is a crossing edge $xy$. It must lie between $C_{1}$ and $C_{2}$; suppose $x\in V(C_{1})$ and $y\in V(C_{2})$.(see Figure 3.8) We can assume that the bold cycle $C'$ is the smallest possible, which is then be an induced cycle. Now observe that $C+w-x_{3}$ is a closed surface; if $R$ be the refinement of the cycle containing the vertices $x_{1},x,y,v,x_{2},w$, then the set $F$ of induced cycles consisting of $C_{1},C_{2},C'$, and the triangle with vertives $u,v,x_{2}$, and those of $R$, forms a closed set and $V(|F|)=V(C+w-x_{3})$. This contradicts the minimality of $S$, and hence $S':=C+w+uv$ must be a trihedron. Note that each induced cycles of $S'$ are induced cycles of $G$. Hence $S'\in S(G)$ by Proposition 4.1.1. Then by the strong minimality of solids and $V(S')=V(S)$, we conclude that $S'=S$. Therefore $S$ is a trihedron. This completes the proof. 
\begin{equation*}
\begin{tikzpicture}
	\filldraw [black] (0,0) circle (1.5pt); \coordinate [label=left:$w$]  (A) at (0,0);
	\draw[thick] (0,0) circle (25pt);
	\filldraw [black] (0.63,0.63) circle (1.5pt); \coordinate [label=right:$x_{1}$]  (A) at (0.7,0.7);
	\filldraw [black] (-0.63,0.63) circle (1.5pt); \coordinate [label=left:$x_{3}$]  (A) at (-0.7,0.7);
	\filldraw [black] (0,-0.87) circle (1.5pt); \coordinate [label=below:$x_{2}$]  (A) at (0,-0.9);
	\filldraw [black] (0.65,-0.6) circle (1.5pt);  \coordinate [label=left:$v$]  (A) at (-0.65,-0.6);
	\filldraw [black] (-0.65,-0.6) circle (1.5pt); \coordinate [label=right:$u$]  (A) at (0.65,-0.6);
	\draw[thick] (0.63,0.63) -- (0,0); 	
	\draw[thick] (-0.63,0.63) -- (0,0);
	\draw[thick] (0,0) -- (0,-0.87);
	\draw[ultra thick] (-0.65,-0.6) -- (0.65,-0.6);
	\draw[ultra thick] (0,-0.87) arc (-90:-42:25pt);
	\draw[ultra thick] (0,-0.87) arc (-90:-138:25pt);
	\coordinate [label=right:$C_{1}$]  (A) at (0.2,-0.2);
	\coordinate [label=left:$C_{2}$]  (A) at (-0.2,-0.2);
	\coordinate [label=above:$C_{3}$]  (A) at (0,0.25);
	\coordinate [label=below: Figure 3.7]  (A) at (0,-1.4);
\end{tikzpicture}	
\begin{tikzpicture}
	\coordinate [label=left:$\qquad$]  (A) at (-2,0);
	\filldraw [black] (0,0) circle (1.5pt); \coordinate [label=left:$w$]  (A) at (0,0);
	\draw[thick] (0,0) circle (25pt);
	\filldraw [black] (0.63,0.63) circle (1.5pt); \coordinate [label=right:$x_{1}$]  (A) at (0.7,0.7);
	\filldraw [black] (-0.63,0.63) circle (1.5pt); \coordinate [label=left:$x_{3}$]  (A) at (-0.7,0.7);
	\filldraw [black] (0,-0.87) circle (1.5pt); \coordinate [label=below:$x_{2}$]  (A) at (0,-0.9);
	\filldraw [black] (0.65,-0.6) circle (1.5pt);  \coordinate [label=left:$v$]  (A) at (-0.65,-0.6);
	\filldraw [black] (-0.65,-0.6) circle (1.5pt); \coordinate [label=right:$u$]  (A) at (0.65,-0.6);
	\filldraw [black] (0.81,0.3) circle (1.5pt);  \coordinate [label=left:$y$]  (A) at (-0.81,0.3);
	\filldraw [black] (-0.81,0.3) circle (1.5pt); \coordinate [label=right:$x$]  (A) at (0.81,0.3);
	\draw[ultra thick] (-0.81,0.3) -- (0.81,0.3); 	
	\draw[thick] (0.63,0.63) -- (0,0); 	
	\draw[thick] (-0.63,0.63) -- (0,0);
	\draw[thick] (0,0) -- (0,-0.87);
	\draw[ultra thick] (-0.65,-0.6) -- (0.65,-0.6);
	\coordinate [label=below: Figure 3.8]  (A) at (0,-1.4);
	\draw[ultra thick] (-0.65,-0.6) arc (-138:-200:25pt);
	\draw[ultra thick] (0.65,-0.6) arc (-42:18:25pt);
\end{tikzpicture}	
\end{equation*}
\end{proof}

\begin{theorem}
	Let $G=(V,E)$ be a graph. Then $S(G)=\mathcal{S}(G)$.
\end{theorem}

\begin{proof}
	We have shown that $S(G)\subseteq \mathcal{S}(G)$. Now we show $S(G)\supseteq \mathcal{S}(G)$. Let $S\in \mathcal{S}(G)$ and $v\in V(S)$. If it is the second case of Proposition 4.1.7, then $S$ is a solid in $G$ by Lemma 4.1.1. Hence we may assume the first case; there are two faces $C_{1},C_{2}$ of $S$ such that $C_{1}\cup C_{2}$ is a spanning subgraph of $S$ which is isormorphic to a trihedron graph. Let $P_{1}$ be the induced path in $G$ from $x\in V(C)$ to $y\in V(C_{1})$ such that $P_{1}-x-y$ is vertex disjoint from $C_{1}$ and $C_{1}\cup P_{1}=C_{1}\cup C_{2}$. Let $P_{2},P_{3}$ be the two paths from $x$ to $y$ along $C_{1}$. $S$ may contain edges between the interior points of $P_{1}$ and $P_{2}$ or $P_{1}$ and $P_{3}$. But then it is easy to observe that the possible "crossing edges" must lie between the three neighbors of $x$ or $y$; otherwise, we can obtain a smaller subgraph which is an underlying graph of a closed subset of $C(G)$, by using Proposition 4.1.6. Moreover, it is also easy to observe that there could be at most one crossing edge between the three neighbors of $x$ or $y$. Therefore, the possible cases end up with trihedron, stamp or a prism, according to the number of crossing edges.(0,1, and 2). Hence $S$ is a solid in $G$, and this completes the proof.
\end{proof}

\begin{corollary}
	A minimal closed surface of a graph $G$ is strongly minimal and induced planar subgraph of $G$.
\end{corollary}

\begin{proof}
	Follows from Proposition 4.1.4 and Theorem 4.1.1.
\end{proof}

\subsection{The third graph characteristic and edge contractions}

\begin{definition}
	Let $G=(V,E)$ be a connected graph. We define a function $\beth^{3}$ called \textit{the third graph characteristic} by 
\begin{equation*}
	\beth^{3}(G)=|S(G)|-|C(G)|+|E(G)|-|V(G)|.
\end{equation*}
\end{definition}

\begin{remark}
	Note that $\beth^{3}(K^{r})=\binom{r}{4}-\binom{r}{3}+\binom{r}{2}-\binom{r}{1}$ and it is an increasing function in $r$. Moreover, $\beth^{3}(K^{1})=\beth^{3}(K^{2})=\beth^{3}(K^{3})=\beth^{3}(K^{4})=-1$ and it is strictly increasing for $r\ge 4$. 
\end{remark}

\begin{proposition}
	Let $G=(V,E)$ be a graph with an edge $e=uv$. Then there is an injection $\Psi : S(G/e)\rightarrow S(G)$ such that $V(\Psi(S)/e)=V(S)$. 
\end{proposition}

\begin{proof}
	The injectivity of $\Psi$ follows from the vertex set property $V(\Psi(S)/e)=V(S)$ directly. Now we construct the map $\Psi$. Let $S$ be a solid in $G$. If $S$ does not use the vertex $v_{e}$, then we define $\Phi(S)=S$. Now suppose $S$ uses the vertex $v_{e}$. 
\begin{description}
	\item{1) Pyramid}

Assume $v_{e}$ is an apex of $S$ and write $C=S-v_{e}$. Each neighbors of $v_{e}$ in $C$ in $G/e$ is adjacent to either $u$ or $v$. If either $u$ or $v$ is adjacent to at least three neighbors of $v_{e}$, then define $\Psi(S)=C+u$  or $C+v$, which are pyramid with apex $u$ or $v$. If none of $u$ and $v$ is adjacent to at least three neighbors of $C$ but one of them, say $u$, is adjacent to two nonadjacent neighbors $v_{e}$, then define $\Psi(S)=C+u$, which is a trihedron. Suppose no previous case happens. If each of $u$ and $v$ is adjacent to two adjacent neighbors of $v_{e}$ on $C$, then $\Psi(S):=C+u+v$ is a prism. If only $u$ is adjacent to two adjacent neighbors of $v_{e}$ on $C$, then $v$ is adjacent to exactly one vertex of $C$ since $v_{e}$ has at least three neighbors in $C$; in that case, define $\Psi(S):=C+u+v$, which is a stamp. Note that at least one of $u$ and $v$ must have at least two neighbors in $C$ since $v_{e}$ has more than two neighbors in $C$. This covers all cases when $v_{e}$ is an apex of $S$.

Now suppose $v_{e}$ is not an apex of $S$. Hence $S$ has an apex $w\ne v_{e}$ and $v_{e}$ is a vertex of the cycle $S-w$. Then define $\Phi(S):=\phi(S-w)+w$, which is a pyramid with apex $w$. 

	\item{2) Trihedron}

Suppose $v_{e}$ has degree $2$ in the trihedron $S$. Let $C$ be one of the two faces of $S$ that contains $v_{e}$, and let $P$ be the induced path from $C$ to itself such that $S=C\cup P$. Then define $\Psi(S):=\psi(C)\cup P$, which is a trihedron. 

Assume $v_{e}$ has degree $3$ in $S$. Let $C$ and $P$ be the face and path in $S$ as before, and let $x$ be the neighbor of $v_{e}$ in $P$. If $\psi(C)=C-v_{e}+u+v$, then define $\Psi(S)=\psi(C)+(P-v_{e})$, which is either a stamp or trihedron. Otherwise, we may assume that $\psi(C)=C-v_{e}+u$. If $x$ is adjacent to $u$ in $G$, then $\Psi(S):=\psi(C)+(P-v_{e})$ is a trihedron. If not, then $x$ should be adjacent to $v$. If $v$ is adjacent to both the two neighbors of $u$ on $\psi(C)$, then $\Psi(S):=S-v_{e}+v$ is a trihedron, which is isomorphic to $S$. Otherwise, $\Psi(S):=\psi(C)+(P-v_{e})+v$ is either a pyramid with apex $v$(if $P$ has length 1) or a stamp or trihedron. 

	\item{3) Stamp}

Suppose $v_{e}$ is a vertex of the triagle $T$ in $S$. Let $P_{1},P_{2},P_{3}$ be the paths from the three vertices of $T$ to the same endpoint and suppose $v_{e}$ is the initial point of $P_{1}$. Let $x$ be the neighbor of $v_{e}$ in $P_{1}$. Assume $\psi(T)=T-v_{e}+u+v$. We may suppose $x$ is adjacent to $u$. Then we define $\Psi(S)=S-v_{e}+u$, which is a trihedron. Otherwise, without loss of generality, suppose $\psi(T)=T-v_{e}+u$. If $ux\in E$, then $\Psi(S):=S-v_{e}+u$ is isomorphic to $S$. If $ux\ne E$, then $xv\in E$. Note that $\Gamma_{S-v_{e}}(v)\subseteq \{V(T-u),x\}$. If $v$ is adjacent to both of the two vertices of $T-v$, then we may define $\Psi(S)=S-v_{e}+v$, which is isomorphic to $S$. If $v$ is adjacent to only one vertex of $T-u$, then we define $\Psi(S):=S-v_{e}+v$, which is a trihedron. Otherwise, $\Gamma_{S-v_{e}}(v)=\{x\}$ and we define $\Psi(S):=S-v_{e}+u+v$, which is a stamp.

Suppose $v_{e}$ is not in $T$ but has degree $3$. This case the construction is similar to that for trihedron. If either $u$ or $v$ is adjacent to all of three neighbors of $v_{e}$ in $S$, then we replace $v_{e}$ by such vertex obtaining an isomorphic stamp $\Psi(S)$ in $G$. Otherwise, let $C$ be the face of $S$ using the paths $P_{2}$ and $P_{3}$ and define $\Psi(S):=\psi(C)+(P_{1}-v_{e})$; it could be a stamp or a prism. 

Suppose $v_{e}$ has degree $2$. Assume $P_{1}$ uses $v_{e}$ and let $C$ be the face of $S$ containing the paths $P_{1}$ and $P_{2}$. Define $\Psi(S)=\psi(C)+P_{3}$, which is a stamp. 

	\item{4) Prism}

Construction is similar to that for stamps. 

\end{description}
\end{proof}

\begin{lemma}
	Let $G=(V,E)$ be a graph. Then we have 
\begin{equation*}
	|S(G)|-|S(G/e)| \ge |C(G)| - |C(G/e)| - |C^{3}_{e}(G)|. 
\end{equation*}
\end{lemma}

\begin{proof}
	Any induced cycle in $G/e$ that does not use the vertex $v_{e}$ is in the image of $\psi$. We partition the set $C(G)$ of induced cycle in $G$ into two classes, $\mathcal{C}$ and $\mathcal{C}'$, where $\mathcal{C}$ is the set of induced cycles $C$ such that the induced subgraph of $G/e$ with vertex set $V(C/e)$ is an induced cycle, and $\mathcal{C}'$ is its compliment. 

Observe that the triangles using the edge $e$ are in the class $\mathcal{C}'$. Note that for any induced cycle $C$ in $G$ of length $\ge 4$ using the edge $e$, $C/e$ is an induced cycle in $G/e$. Hence if $C'\in \mathcal{C}'\setminus C^{3}_{e}(G)$, then $C'$ does not use the edge $e$. Also $C'$ must use either $u$ or $v$, since otherwise $C/e=C\in C(G)$. Thus we conclude that $C'$ must contain at least two nonadjacent neighbors of $u$ or $v$ for the edge contraction of $e$ to violate the inducedness of $C'$. We suppose $C'$ contains such neighbors of $u$. Then $C'+u$ is either a trihderon or a pyramid, and moreover it is not in the image of $\Psi$ since the induced subgraph of $G/e$ with vertex set $V((C'+u)/e)$ is not a solid graph; it is a cycle with some crossing edges, where the crossing edges "does not cross each other." Hence for the class $\mathcal{C}'$ we have found at least $|\mathcal{C}'|-|C^{3}_{e}(G)|$ distinct solids in $G$ that are not in the image of $\Psi$.

Now we consider the class $\mathcal{C}$. We give an equivalence relation on this class by defining $C\sim C'$ if $C/e=C'/e$. Fix an equivalence class $[C]$. Let $C-u-v$ be the induced path $Q$ from $x$ to $y$. Then observe that $[C]=\{Q+u,Q+v\}$ if both $u$ and $v$ are adjacent to both $x$ and $y$, and $|[C]|=1$ otherwise. If $[C]=\{Q+u,Q+v\}$, then we correspond a pyramid $S:=Q+u+v$ with vertex $u$. Notice that $S$ is not in the image of $\Psi$ since $S/e$ is an induced cycle. Hence for each equivalence class $[C]$, we can find $N$ solids in $G$ that is not in the image of $\Psi$ with $N\ge |[C]|-1$. Moreover, if $[C']$ is another equivalence class with size $2$, then the associated solid $S':=Q'+u+v$ where $Q'=C'-u-v$ is different from $S$ since $Q'\ne Q$. Therefore summing over all equivalence classes, we obtain $|\mathcal{C}|-|\mathcal{C}/\sim|$ distinct solids of $S(G)\setminus \text{im}\,\Psi$. Now notice that for each equivalence class $[C]$, there is exactly one induced cycle in the class that is in the image of $\psi$; hence $|\text{im}\,\psi|\ge |\mathcal{C}/\sim|$. Therefore we obtain
\begin{eqnarray*}
	|S(G)|-|S(G/e)|&=& |S(G)|-|\text{im}\,\Psi|\\
&\ge & (|\mathcal{C}'|-|C^{3}_{e}(G)|)+(|\mathcal{C}|-|\mathcal{C}'/\sim|)\\
&\ge & |C(G)| - |C^{3}_{e}(G)|- |\text{im}\,\psi|\\
&=& |C(G)| - |C(G/e)| - |C^{3}_{e}(G)|. 
\end{eqnarray*}
\end{proof}

\begin{theorem}
	Let $G=(V,E)$ be a graph with an edge $e$. Then we have 
\begin{equation*}
\beth^{3}(G/e)\le \beth^{3}(G).
\end{equation*}
\end{theorem}

\begin{proof}
Lemma 4.2.1 and Proposition 2.1.1 yields 
\begin{eqnarray*}
	\beth^{3}(G)-\beth^{3}(G/e) &=& \Delta S - \Delta C + \Delta E -\Delta V\\
&\ge & -|C^{3}_{e}(G)|+(|C^{3}_{e}(G)|+1) -1 \ge 0.
\end{eqnarray*}
\end{proof}

\begin{corollary}
	Let $G=(V,E)$ be a connected graph. Then $\beth^{3}(K^{\chi(G)})\le \beth^{3}(G)$. That is, 
\begin{equation*}
\binom{h(G)}{4}-\binom{h(G)}{3}+\binom{h(G)}{2}-\binom{h(G)}{1} \le \beth^{3}(G).
\end{equation*}
\end{corollary}

\begin{proof}
	Follows from Theorem 4.2.1. The proof is similar to that of Corollary 2.2.1. 
\end{proof}

\subsection{The third graph characteristic and nonedge contractions}

We start by the similar "cone relation". 

\begin{proposition}
	Let $G^{w}$ be a cone graph. Then one has $ |S(G^{w})|=|S(G)|+|C(G)|$.
\end{proposition}

\begin{proof}
	There are $|S(G)|$ solids in $G^{w}$ that does not use the vertex $w$. Now we show there are $|C(G)|$ solids in $G$ that use the vertex $w$. At first, observe that for each $C\in C(G)$, the cone $C^{w}$ is an solid in $G$. On the other hand, let $S$ be an solid in $G^{w}$ that uses $w$. Then exactly one face of $S$ is contained in $G$; for, at least one face is contained in $G$, and if two faces $C_{1}$ and $C_{2}$ of $S$ are contained in $G$, then the cones $C_{1}^{w}$ and $C_{2}^{w}$ are properly contained in $S$, which contradicts the minimality of $S$. Hence there are exactly $|C(G)|$ solids in $G$ that uses the vertex $w$, and this shows the assertion. 
\end{proof}

\begin{proposition}
	 Let $G=(V,E)$ be a connected graph with a vertex $w$. Then there is an injection $\Phi : S(G/w)\setminus S(K_{w}) \rightarrow S(G)\setminus S(G_{w})$ such that $V(\Phi(S)/w-w)=V(S)$ for all $S\in S(G/w)$. 
\end{proposition}

\begin{proof}
	Note that injectivity of $\Phi$ follows from the property $V(\Phi(S)/w-w)=V(S)$, for if $\Phi(S)=\Phi(S')$, then we have $V(S)=V(S')$ and hence $S=S'$ since two induced subgraph of $G$ with the same vertex set is identical. Hence it suffices to construct a the map $\Phi$ satisfying the vertex set property. 

We may assume $G_{w}$ is not a complete graph, since otherwise $G/w=G$. Let $c$ be a proper colouring of $G/w$, so that the complete graph $K_{w}$ gets colours $\{1,2,\cdots,k\}$, with $c(w)=k$. We may extend $c$ to a colouring $\overline{c}$ of $G$, that is, $\overline{c}(v)=c(v/w)$. 

If $S$ uses no vertex of $K_{w}$, then define $\Phi(S):=S$. Now we may assume $S$ uses at least one vertex of $K_{w}$. Let $H:=G_{w}-w$, $K=K_{w}-w$ as before. Let $S\in S(G/w)\setminus S(K_{w})$. Note that $S$ can use at most three vertices of $K_{w}$ since otherwise it must be a $K^{4}$ in $K_{w}$ by the minimality(Proposition 4.1.4). Hence $S\cap K_{w}$ could be $K^{1}$, $K^{2}$ or $K^{3}$. Observe that $S$ cannot use the vertex $w$; for, if $S$ does use $w$, then it must use at least two neighbors of $w$ in $H$, say, $z_{1}$ and $z_{2}$. Then $S$ contains a triangle $T$ with vertices $w,z_{1},z_{2}$, and hence $S$ is either a pyramid, stamp, or prism. Note that in any case, any vertex of $S$ contained in a triangle has degree $3$, which is clear from their graphs. In particular, $w$ has degree 3 in $S$, and therefore there is another neighbor $z_{3}\in \Gamma_{S}(w)\cap K$. This yields a contradiction to preceeding argument, since then $S$ uses four vertices $w,z_{1},z_{2},z_{3}$ of $K_{w}$. Hence $w\notin V(S)$. Later on, we will denote $\Gamma_{K_{w}}(S)\subseteq \{z_{1},z_{2},z_{3}\}$, and $c(z_{i})=i$. 

In the case of non-pyramid solid, there is a case for which the construction is just the subdivision of an edge that does not belong to a triangle, which sends such a solid to the same type of solid. It is the case that $S$ has one or two vertices $z_{1},z_{2}\in K_{w}$ and those vertices have only two neighbors in $S$. In this case, take any face $C$ of $S$ that uses $u_{1}u_{2}$, and replace $C$ by $\psi(S)$. The resulting graph is isomorphic to $S$ itself or a subdivision of an edge that does not belong to a triangle. 

Now there are 21 cases regarding the types of solid and there positions within $K_{w}$. Except the easy one, other 20 cases can be dealt with in the same manner. We will demonstrate the construction for some examplary cases, and substitute a written proof to the table at the end of this paper, where the map $\Phi$ is described for all cases pictorially. 

First, consider $S$ a pyramid in $G/w$ with apex $z_{1}\in K$. Additionally, assume $S$ is not a $K_{4}$, i.e., the induced cycle $S-z_{1}$ has length $\ge 4$. Denote $\{x_{1},\cdots,x_{k}\}=\Gamma_{S}(z_{1})$. If $H$ has some vertex $v$ with $\overline{c}(v)=1$ and adjacent to two vertices of $\Gamma_{S}(z_{1})$ which are not adjacent in $S$, (and hence in $G$), then $\Phi(S):=S-z_{1}+v$ is a trihedron or a pyramid. Otherwise, choose $x_{i},x_{j}\in \Gamma_{S}(z_{1})$ such that $x_{i}x_{j}\notin E(S)$, and choose $v_{i},v_{j}\in V(H)$ such that $\overline{c}(v_{i})=\overline{c}(v_{j})=1$ and $v_{i}x_{i},v_{j}x_{j}\in E(G)$. Then $\Phi(S):=S-z_{1}+v_{i}+v_{j}+w$ is a trihedron or a stamp. If $S$ were $K_{4}$ in the preceeding paragraph, then let $C$ be any triangle of $S$ using $z_{1}$ and consider $\phi(C)$, which can be of type 1 or 2.(recall remark 3.2.2.) Suppose $\phi(C)$ is type 1, i.e., a triangle. Let $v$ be the vertex of $\phi(C)$ in $H$. If $S-z_{1}+v$ is $K_{4}$, we define $\Phi(S)=S-z_{1}+v$. If not, choose a neighbor $v_{1}\in V(H)$ of the vertex $S-C$ such that $\overline{c}(v_{1})=1$ and define $\Phi(S)=S-z_{1}+\phi(C)+v_{1}+w$, which is a stamp. Now if $\phi(C)$ is type 2, which uses two vertices of $H$, say $v_{1},v_{2}$, then choose a neighbor $v_{3}\in V(H)$ of the vertex $S-C$ such that $\overline{c}(v_{3})=1$ and define $\phi(S)=S-z_{1}+\phi(C)+v_{3}$, which is also a stamp. 

Let $S$ be a trihedron. Since it has no triangle, it has one or two vertices in $K_{w}$. If $V(S\cap K_{w})=\{z_{1}\}$, then we only need to consider $z_{1}$ with degree $3$, by preceeding paragraph. Denote $\Gamma_{S}(z_{1})=\{x_{1},x_{2},x_{3}\}$. If $G_{w}$ has a vertex $z$ such that $\overline{c}(z)=c(z)$ and $\Gamma_{S-z_{1}}(z)=\Gamma_{S}(z_{1})$, then we define $\Phi(S)=S-z_{1}+z\approx S$. Suppose there is no such vertex $z$ in $G_{w}$. Let $C$ be the face of $S$ using $z_{1},x_{1}$, and $x_{2}$. Let $z\in G_{w}$ be any vertex such that $\overline{c}(z)=z_{1}$ and $zx_{1}\in E(G)$. Then we define $\Phi(S)=\psi(C)+(S-C)+z$, which is a trihedron. Now suppose $V(S\cap K_{w})=\{z_{1},z_{2}\}$. We may assume $z_{1}$ has three neighbors in $S$. Let $C$ be a face of $S$ using the edge $z_{1}z_{2}$. $\phi(C)$ can be of type 3 or 4. Let $P$ be the induced path of $S$ such that $C\cup P=S$. If $\phi(C)$ is of type $3$ and $\phi(C)+P$ is isomorphic to $S$, the we define $\Phi(S):=\phi(C)+P$. If $\phi(C)+P$ is not isomorphic to $S$, then choose a colour 1 vertex $v$ in $H$ which is adjacent to the neighbor of $z_{1}$ in $P$, and define $\Phi(S):=\phi(C)+v+w$. On the other hand, consider $\phi(C)$ is of type $4$. If $\phi(C)+(P-z_{1})$ is a trihedron, we define it as $\Phi(S)$. Otherwise, choose the same colour 1 vertex $v\in H$ as before, and define $\Phi(S):=\phi(C)+(P-z_{1})+v$. This covers all cases for trihdron. 

We have demonstrated how to consturct $\Phi$ for four cases, two for pyramid and two for trihedron, which are dipicted in the table on the first two rows for pyramid and trihdron. Note that, except the first case, we used the similar method for the other three; choose a face $C$ that uses all vertices of $S\cap K$, and consider the image of $\phi(C)$. Use the principle that any vertex of $S$ adjacent to $z_{i}$ is adjacent to some vertex in $H$ of colour $i$, and possibly use $w$ as well, which is adjacent to all vertices of $H$. The consturction for remaining cases dipicted in the table uses this stretage. 

For the last comment, here is how to read the construction table. The shaded first column shows possible cases of solids in $G/w$ that are subject to the map $\Phi$ and the possible images are described in the same row. Note that those simple cases for non-pyramid solids $S$ described in the proof of Proposition 4.3.3, where $\Phi(S)$ can be defined by a subdivision, is not dipicted. Bold cycles are choosen to be $C$, which are subject to the map $\phi$, dotted edges can either be edges or nonedges, and numbers next to the vertices of $H$ denotes their colour with the colouring $\overline{c}$.
\end{proof}

\begin{remark}
	If $S$ is a solid in $G$ such that $S/w-w$ is not a solid in $G/w$. Then $S$ is not in the image of $\Phi$. 
\end{remark}

\begin{lemma}
	Let $G=(V,E)$ be a connected graph with a vertex $w$. Then
\begin{equation*}
|S(G)\setminus S(G_{w})|-|S(G/w)\setminus S(K_{w})|\ge |C(G)\setminus C(G_{w})| -|C(G/w)\setminus C(K_{w})|.
\end{equation*}
\end{lemma}

\begin{proof}
Proposition 3.2.3 gives us the following partition 
\begin{equation*}
	C(G)=\mathcal{C}_{1}\sqcup \mathcal{C}_{2} \sqcup \mathcal{C}_{3},
\end{equation*}
where $\mathcal{C}_{1}$ corresponds to the case (i), $\mathcal{C}_{2}$ to $(ii)$ and so on. Notice that any $C\in \mathcal{C}_{2}\cup \mathcal{C}_{3}$ does not use $w$ and $C+w$ is either a triheron or pyramid in $G$. Moreover, such solids $C+w$ is not in the image of $\Phi$ since $V((C+w)/w-w)$ is not vertex set of any solid in $G$; in either cases, $C$ is the union of induced paths $P_{1},\cdots,P_{k}$ with end points in $H$ and mutually distinct internal vertices, so that the induced subgraph of $G/w$ with vertex set $V((C+w)/w-w)$ is complete graph $K$ with $\le 3$ vertices together with induced paths with endpoints on $K$ and with mutually distinct internal vertices, which are clearly not a solid. Hence the set $\{C+w\,|\, C\in \mathcal{C}_{2}\cup \mathcal{C}_{3}\}$ of solids in $G$ is disjoint from the image of $\Phi$. 

On the other hand, let $\mathcal{C}_{1}^{i}$ be the set which consists of induced cycles $C$ in $\mathcal{C}_{1}$ that is of type $i$, for $0\le i \le 4$. Then what we are going to show is that we can find at least $|\mathcal{C}_{1}\setminus \mathcal{C}_{1}^{0}|$ distinct solids in $G$ which are not in the image of $\Phi$. Let $c$ be a proper colouring of $G/w$, and get an extended colouring $\overline{c}$ of $G$. Then notice that the sets $\mathcal{C}_{1}^{1}\cup \mathcal{C}_{1}^{2}$ and $\mathcal{C}_{1}^{3}\cup \mathcal{C}_{1}^{4}$ are disjoint, since they use different set of colours in $H$. 

\begin{description}
	\item{Case 1.} $\mathcal{C}_{1}^{1}\cup \mathcal{C}_{1}^{2}$

Give an equivalence relation on $\mathcal{C}_{1}^{1}\cup \mathcal{C}_{1}^{2}$ by defining $C\sim C'$	 if $C/w-w=C'/w-w$. That is, two induced cycles are equivalent if the induced subpaths "strictly below" the cone $H^{w}$ are the same and both cycles use the same colour class in $H$. Fix a nonempty equlvance class $[C]$; that is, fix an induced path $Q$ from $x$ to $y$ below the cone $H^{w}$. We may assume that the cycles in the class $[C]$ uses colour $1$ vertices of $H$. First, suppose $x=y$, that is, the induced path $Q$ is of length $0$. Let $A$ be the set of colour 1 vertices of $H$ that is adjacent to $x$. Then we have $|[C]|=\binom{|A|}{2}$, and for each three distinct vertices $z_{1},z_{2},z_{3}$ of $A$ we correspond a trihedron $S:=x+w+z_{1}+z_{2}+z_{3}$. Note that $S$ is not in the image of $\Phi$, since $S/w$ has three vertices. Hence there are at least $\binom{|A|}{3}$ distinct solids in $G$ that is not in the image of $\Phi$ corresponding to the class $[C]$. Note that $\binom{|A|}{3}\ge \binom{|A|}{2}-1=|[C]|-1$. 

Now suppose $x\ne y$. Let $\mathcal{P}_{x}$ be the set of length 2 path from $x$  to $w$ such that the interior point is not adjacent to $y$, and define $\mathcal{P}_{y}$ similarly. Also define $U$ to be the set of colour 1 vertices of $H$ that are adjacent to both $x$ and $y$. Denote $p_{x}:=|\mathcal{P}_{x}|$ and $p_{y}:=|\mathcal{P}_{y}|$. Then observe that the number $|[C]|$ of induced cycles in the class $[C]$ can be written as 
\begin{equation*}
	|[C]| = p_{x}p_{y} + |U|. 
\end{equation*}
Now for each triple $(P_{1},P_{2},P_{3})\in \mathcal{P}_{1}\times \mathcal{P}_{1}\times \mathcal{P}_{2}$ such that the three paths use different vertices in $H$, we correspond a trihedron $S:=Q\cup \bigcup_{i=1}^{3}P_{i}$ in $G$. We can associate trihedrons for each triple of $\mathcal{P}_{1}\times \mathcal{P}_{2} \times \mathcal{P}_{2}$ similarly. On the other hand, for each quadruple $(P_{1},P_{2},P_{3},P_{4})\in \mathcal{P}_{1}^{2}\times \mathcal{P}_{2}^{2}$ such that $P_{2}$ and $P_{3}$ use the same vertex in $H$ and the four paths use three vertices in $H$, we correspond a pyramid $S:=Q\cup \bigcup_{i=1}^{4}P_{i}$ in $G$. Lastly, for each pair $(u_{1},u_{2})\in U^{2}$ with $u_{1}\ne u_{2}$, we correspond a trihedron $S:=Q+u_{1}+u_{2}$. Those four solids are described below: 
\begin{equation*}
\begin{tikzpicture}
	\draw (-1.4,0.04) -- (0,1.2) -- (1.4,0.04);
	\filldraw [black] (0,1.2) circle (1.5pt); \coordinate [label=left:$w$]  (A) at (0,1.2);
	\draw (0,0) ellipse (40pt and 10pt);
	\coordinate [label=left:$H$]  (A) at (-1.5,0);
	\draw[ultra thick] (0,1.15) -- (0.2,0) -- (0.6,-0.5) -- (0.6,-1) -- (0.3,-1.35) -- (-0.3,-1.35) -- (-0.6, -1) -- (-0.6,-0.5) -- (-0.2,0) -- (0,1.15);
	\filldraw [black] (0.6,-0.5) circle (1.5pt); 
	\filldraw [black] (-0.6,-0.5) circle (1.5pt); 	
	\filldraw [black]  (0.2,0) circle (1.5pt); 	\coordinate [label=right:$1$]  (A) at (0.2,0.1);
	\filldraw [black]  (-0.2,0) circle (1.5pt); 	\coordinate [label=left:$1$]  (A) at (-0.2,0.1);
	\filldraw [black]  (-0.7,0) circle (1.5pt); 	\coordinate [label=left:$1$]  (A) at (-0.7,0.1);
	\draw[ultra thick] (0,1.15) -- (-0.7,0) -- (-0.6,-0.5);
	\coordinate [label=below:Figure 4.1]  (A) at (0,-1.5);
	\coordinate [label=left:$x$]  (A) at (-0.6,-0.5);
	\coordinate [label=right:$y$]  (A) at (0.6,-0.5);
\end{tikzpicture}	
\begin{tikzpicture}
	\draw (-1.4,0.04) -- (0,1.2) -- (1.4,0.04);
	\filldraw [black] (0,1.2) circle (1.5pt); \coordinate [label=left:$w$]  (A) at (0,1.2);
	\draw (0,0) ellipse (40pt and 10pt);
	\coordinate [label=left:$\quad H$]  (A) at (-1.5,0);
	\draw[ultra thick] (0,1.15) -- (0.2,0) -- (0.6,-0.5) -- (0.6,-1) -- (0.3,-1.35) -- (-0.3,-1.35) -- (-0.6, -1) -- (-0.6,-0.5) -- (-0.2,0) -- (0,1.15);
	\filldraw [black] (0.6,-0.5) circle (1.5pt); 
	\filldraw [black] (-0.6,-0.5) circle (1.5pt); 	
	\filldraw [black]  (0.2,0) circle (1.5pt); 	\coordinate [label=right:$1$]  (A) at (0.2,0.1);
	\filldraw [black]  (-0.2,0) circle (1.5pt); 	\coordinate [label=left:$1$]  (A) at (-0.2,0.1);
	\filldraw [black]  (0.7,0) circle (1.5pt); 	\coordinate [label=right:$1$]  (A) at (0.7,0.1);
	\draw[ultra thick] (0,1.15) -- (0.7,0) -- (0.6,-0.5);
	\coordinate [label=below:Figure 4.2]  (A) at (0,-1.5);
	\coordinate [label=left:$x$]  (A) at (-0.6,-0.5);
	\coordinate [label=right:$y$]  (A) at (0.6,-0.5);
\end{tikzpicture}	
\begin{tikzpicture}
	\draw (-1.4,0.04) -- (0,1.2) -- (1.4,0.04);
	\filldraw [black] (0,1.2) circle (1.5pt); \coordinate [label=left:$w$]  (A) at (0,1.2);
	\draw (0,0) ellipse (40pt and 10pt);
	\coordinate [label=left:$\quad H$]  (A) at (-1.5,0);
	\draw[ultra thick] (0,1.15) -- (0,0) -- (0.6,-0.5) -- (0.6,-1) -- (0.3,-1.35) -- (-0.3,-1.35) -- (-0.6, -1) -- (-0.6,-0.5) -- (0,0) -- (0,1.15);
	\filldraw [black] (0.6,-0.5) circle (1.5pt); 
	\filldraw [black] (-0.6,-0.5) circle (1.5pt); 	
	\filldraw [black]  (0,0) circle (1.5pt); 	\coordinate [label=right:$1$]  (A) at (0,0.1);
	\filldraw [black]  (-0.7,0) circle (1.5pt); 	\coordinate [label=left:$1$]  (A) at (-0.7,0.1);
	\filldraw [black]  (0.7,0) circle (1.5pt); 	\coordinate [label=right:$1$]  (A) at (0.7,0.1);
	\draw[ultra thick] (0,1.15) -- (-0.7,0) -- (-0.6,-0.5);
	\draw[ultra thick] (0,1.15) -- (0.7,0) -- (0.6,-0.5);
	\coordinate [label=below:Figure 4.3]  (A) at (0,-1.5);
	\coordinate [label=left:$x$]  (A) at (-0.6,-0.5);
	\coordinate [label=right:$y$]  (A) at (0.6,-0.5);
\end{tikzpicture}	
\begin{tikzpicture}
	\draw (-1.4,0.04) -- (0,1.2) -- (1.4,0.04);
	\filldraw [black] (0,1.2) circle (1.5pt); \coordinate [label=left:$w$]  (A) at (0,1.2);
	\draw (0,0) ellipse (40pt and 10pt);
	\coordinate [label=left:$\quad H$]  (A) at (-1.5,0);
	\draw[ultra thick] (0,-0.1) -- (0.6,-0.5) -- (0.6,-1) -- (0.3,-1.35) -- (-0.3,-1.35) -- (-0.6, -1) -- (-0.6,-0.5) -- (0,-0.1);
	\filldraw [black] (0.6,-0.5) circle (1.5pt); 
	\filldraw [black] (-0.6,-0.5) circle (1.5pt); 	
	\filldraw [black]  (0,-0.1) circle (1.5pt); 	\coordinate [label=below:$1$]  (A) at (0,-0.1);
	\filldraw [black]  (0,0.2) circle (1.5pt); 	\coordinate [label=right:$1$]  (A) at (0,0.2);
	\draw[ultra thick] (-0.6,-0.5) -- (0,0.2) -- (0.6,-0.5);
	\coordinate [label=below:Figure 4.4]  (A) at (0,-1.5);
	\coordinate [label=left:$x$]  (A) at (-0.6,-0.5);
	\coordinate [label=right:$y$]  (A) at (0.6,-0.5);
\end{tikzpicture}	
\end{equation*}
Note that the four types of solids are all distinct. Moreover, each associated solid $S$ is not in the image of $\Phi$ since $S/w-w\in C(G)$ and hence there is no solid in $G$ with vertex set $V(S/w-w)$.(Recall the property of the map $\Phi$.) The number of the corresponding four types of solids are $\binom{p_{x}}{2}\cdot p_{y}$, $\binom{p_{y}}{2}\cdot p_{x}$, $|U|p_{x}p_{y}$ and $\binom{|U|}{2}$, and summing them, we get $N$ distinct solids of $S(G)\setminus S(G_{w})\setminus \text{im}\,\Phi$ where $N:=p_{x}p_{y}\left( \frac{1}{2}(p_{x}+p_{y})-1+|U|\right) +\binom{|U|}{2}$. Then I claim that 
\begin{equation*}
N=p_{x}p_{y}\left( \frac{1}{2}(p_{x}+p_{y})-1+|U|\right) +\binom{|U|}{2}\ge p_{x}p_{y}+|U|-1=|[C]|-1.
\end{equation*}
Observe that $\binom{|U|}{2}\ge |U|-1$ for all $|U|$; hence the claim is true if $p_{x}p_{y}=0$. Also note that the assertion holds if $|U|\ge 2$ or $p_{1}+p_{2}\ge 4$. Hence we may assume $p_{x},p_{y}\ge 1$, $p_{x}+p_{y}\le 3$ and $|U|\le 1$. Now checking on the remaining cases is straightforward; $p_{x}=p_{y}=1$ yields the RHS$\le 0$ while $N\ge 0$. Letting $p_{x}=1$, $p_{y}=2$ and $|U|=0$ yields LHS$=$RHS$=1$, and chaning $|U|$ from $0$ to $1$ yields LHS$=3$ and RHS=$2$. This shows the claim. Therefore we can find at least $|[C]|-1$ solids of $S(G)\setminus S(G_{w})\setminus \text{im}\,\Phi$ for each equivalence class $[C]$, and the associated solids for each class $[C]$ are distinct. Hence summing over all classes, we obtain at least $|\mathcal{C}_{1}^{1}\cup \mathcal{C}_{1}^{2}| - |\mathcal{C}_{1}^{1}\cup \mathcal{C}_{1}^{2}/\sim|$ such solids.

	\item{Case 2.} $\mathcal{C}_{1}^{3}\cup \mathcal{C}_{1}^{4}$

Define the equivalence relation $\sim$ in the same way as in previous case. Fix a nonempty class $[C]$, with the induced path $Q$ from $x$ to $y$ as before. We may assume that $x$ is adjacent to colour 1 vertices, $y$ is adjacent to colour 2 vertices of $H$. Note that neither $x$ or $y$ is adjacent to both colour 1 and 2 vertices. Define $\mathcal{P}_{x}$ be the set of length 2 paths from $x$ to $w$, and define $\mathcal{y}$ be similarly. Denote $p_{x}:=|\mathcal{P}_{x}|$ and $p_{y}=|\mathcal{P}_{y}|$. Notice that if $C'\in [C]$, then either $C$ uses two adjacent vertices $z_{1}$, $z_{2}$(type 3) or $C$ uses two nonadjacent vertices $z_{1}$ and $z_{2}$ with different colours \textit{and} $w$.(type 4). We may suppose $\overline{c}(z_{i})=i$ for $i=1,2$. Then it is obvious that each cycle $C'\in [C]$ in the class correponds to the pair $(xz_{1}w,yz_{2}w)\in \mathcal{P}_{1}\times \mathcal{P}_{2}$; hence we have $|[C]|=p_{x}p_{y}$. Now for each triple $(P_{1},P_{2},P_{3})\in \mathcal{P}_{x}\times \mathcal{P}_{x}\times \mathcal{P}_{y}$, we correspond a solid $S$ in $G$, which is an induced subgraph of $G$ with vertex set $V(P_{1}\cup P_{2}\cup P_{3} \cup Q)$. There are three different types of $S$, in repect to the number of edges between the two colour 1 vertices and one colour 2 vertex in $H$ that are used in the paths $P_{1}$, $P_{2}$ and $P_{3}$. Below are the graphs of $S$ corresponding to the number of such edges(0, 1 and 2 from left to right). For the exceptional case when $x=y$ and the colour $2$ vertex is adjacent to the two colour 1 vertices, we correspond a pyramid as in figure 8.  
\begin{equation*}
\begin{tikzpicture}
	\draw (-1.4,0.04) -- (0,1.2) -- (1.4,0.04);
	\filldraw [black] (0,1.2) circle (1.5pt); \coordinate [label=left:$w$]  (A) at (0,1.2);
	\draw (0,0) ellipse (40pt and 10pt);
	\coordinate [label=left:$H$]  (A) at (-1.5,0);
	\draw[ultra thick] (0,1.15) -- (0.2,0) -- (0.6,-0.5) -- (0.6,-1) -- (0.3,-1.35) -- (-0.3,-1.35) -- (-0.6, -1) -- (-0.6,-0.5) -- (-0.2,0) -- (0,1.15);
	\filldraw [black] (0.6,-0.5) circle (1.5pt); 
	\filldraw [black] (-0.6,-0.5) circle (1.5pt); 	
	\filldraw [black]  (0.2,0) circle (1.5pt); 	\coordinate [label=right:$2$]  (A) at (0.2,0.1);
	\filldraw [black]  (-0.2,0) circle (1.5pt); 	\coordinate [label=left:$1$]  (A) at (-0.2,0.1);
	\filldraw [black]  (-0.7,0) circle (1.5pt); 	\coordinate [label=left:$1$]  (A) at (-0.7,0.1);
	\draw[ultra thick] (0,1.15) -- (-0.7,0) -- (-0.6,-0.5);
	\coordinate [label=below:Figure 4.5]  (A) at (0,-1.5);
	\coordinate [label=left:$x$]  (A) at (-0.6,-0.5);
	\coordinate [label=right:$y$]  (A) at (0.6,-0.5);
\end{tikzpicture}
\begin{tikzpicture}
	\draw (-1.4,0.04) -- (0,1.2) -- (1.4,0.04);
	\filldraw [black] (0,1.2) circle (1.5pt); \coordinate [label=left:$w$]  (A) at (0,1.2);
	\draw (0,0) ellipse (40pt and 10pt);
	\coordinate [label=left:$\quad H$]  (A) at (-1.5,0);
	\draw[ultra thick] (0,1.15) -- (0.2,0) -- (0.6,-0.5) -- (0.6,-1) -- (0.3,-1.35) -- (-0.3,-1.35) -- (-0.6, -1) -- (-0.6,-0.5) -- (-0.7,0) -- (0,1.15) -- (-0.2,0);
	\filldraw [black] (0.6,-0.5) circle (1.5pt); 
	\filldraw [black] (-0.6,-0.5) circle (1.5pt); 	
	\filldraw [black]  (0.2,0) circle (1.5pt); 	\coordinate [label=right:$2$]  (A) at (0.2,0.1);
	\filldraw [black]  (-0.2,0) circle (1.5pt); 	\coordinate [label=left:$1$]  (A) at (-0.2,0.1);
	\filldraw [black]  (-0.7,0) circle (1.5pt); 	\coordinate [label=left:$1$]  (A) at (-0.7,0.1);
	\draw[ultra thick] (-0.6,-0.5) -- (-0.2,0) -- (0.2,0);
	\coordinate [label=below:Figure 4.6]  (A) at (0,-1.5);
	\coordinate [label=left:$x$]  (A) at (-0.6,-0.5);
	\coordinate [label=right:$y$]  (A) at (0.6,-0.5);
\end{tikzpicture}
\begin{tikzpicture}
	\draw (-1.4,0.04) -- (0,1.2) -- (1.4,0.04);
	\filldraw [black] (0,1.2) circle (1.5pt); \coordinate [label=left:$w$]  (A) at (0,1.2);
	\draw (0,0) ellipse (40pt and 10pt);
	\coordinate [label=left:$\quad H$]  (A) at (-1.5,0);
	\draw[ultra thick]  (-0.2,-0.2) --(0.4,0.1) -- (0.6,-0.5) -- (0.6,-1) -- (0.3,-1.35) -- (-0.3,-1.35) -- (-0.6, -1) -- (-0.6,-0.5) -- (-0.2,-0.2);
	\filldraw [black] (0.6,-0.5) circle (1.5pt); 
	\filldraw [black] (-0.6,-0.5) circle (1.5pt); 	
	\filldraw [black]  (-0.2,-0.2) circle (1.5pt); 	\coordinate [label=above:$2$]  (A) at (0.4,0.1);
	\filldraw [black]  (0.4,0.1) circle (1.5pt); 		\coordinate [label=below:$1$]  (A) at (-0.2,-0.2);
	\filldraw [black]  (-0.4,0.1) circle (1.5pt); 	 \coordinate [label=above:$1$]  (A) at (-0.4,0.1);
	\draw[ultra thick] (0.6,-0.5) -- (0.4,0.1) -- (-0.4,0.1) -- (-0.6,-0.5);
	\coordinate [label=below:Figure 4.7]  (A) at (0,-1.5);
	\coordinate [label=left:$x$]  (A) at (-0.6,-0.5);
	\coordinate [label=right:$y$]  (A) at (0.6,-0.5);
\end{tikzpicture}
\begin{tikzpicture}
	\draw (-1.4,0.04) -- (0,1.2) -- (1.4,0.04);
	\filldraw [black] (0,1.2) circle (1.5pt); \coordinate [label=left:$w$]  (A) at (0,1.2);
	\draw (0,0) ellipse (40pt and 10pt);
	\coordinate [label=left:$\quad H$]  (A) at (-1.5,0);
	\filldraw [black]  (0.6,0) circle (1.5pt); 		\coordinate [label=right:$1$]  (A) at (0.6,0);
	\filldraw [black]  (-0.6,0) circle (1.5pt); 		 \coordinate [label=left:$1$]  (A) at (-0.6,0);
	\filldraw [black]  (0.1,-0.2) circle (1.5pt); 		 \coordinate [label=left:$2$]  (A) at (0.1,0.1);
	\filldraw [black]  (0,-1) circle (1.5pt); 		\coordinate [label=below:$\text{$x=y$}$]  (A) at (0,-1.2);
	\draw[ultra thick] (0,1.2) -- (0.6,0) -- (0,-1) -- (-0.6,0) -- (0,1.2);
	\draw[ultra thick] (0,1.2) -- (0.1,-0.2) -- (0,-1);
	\draw[ultra thick] (-0.6,0) -- (0.1,-0.2) -- (0.6,0);
	\coordinate [label=below:Figure 4.8]  (A) at (0,-1.5);
\end{tikzpicture}
\end{equation*}
Hence the number of solids we obtain in this way is $\binom{p_{x}}{2}p_{y}$, and by symmetric argument, we obtain another $\binom{p_{y}}{2}p_{x}$. Write $N:=\binom{p_{x}}{2}p_{y}+\binom{p_{y}}{2}p_{x}=\frac{1}{2}p_{x}p_{y}(p_{x}+p_{y})-p_{x}p_{y}$. Moreover, these solids are not in the image of $\Phi$ by the similar reason as in the previous case. 

Now I claim that 
\begin{equation*}
	N:=\frac{1}{2}p_{x}p_{y}(p_{x}+p_{y})-p_{x}p_{y}\ge p_{x}p_{y}.
\end{equation*}
Note that $p_{x},p_{y}\ge 1$ since $[C]$ is nonempty. The inequality holds clearly when $p_{x}+p_{y}\ge 4$, and remaining cases can be checked readily. Therefore, by using similar argument, we obtain at least $|\mathcal{C}_{1}^{3}\cup \mathcal{C}_{1}^{4}|-|\mathcal{C}_{1}^{3}\cup \mathcal{C}_{1}^{4}/\sim|$ distinct solids of $S(G)\setminus \text{im}\,\Phi$. 
\end{description}
Now in both cases, observe that there is at least one cycle $C'\in [C]$ which is in the image of the map $\phi$. That is, $\phi(C/w-w)\in [C]$. Hence 
\begin{equation*}
|\text{im}\, \phi|\ge |\mathcal{C}_{1}^{1}\cup \mathcal{C}_{1}^{4}/\sim|+|\mathcal{C}_{1}^{3}\cup \mathcal{C}_{1}^{4}/\sim|.
\end{equation*}
Observe that the corresponding solids of $\mathcal{C}_{\ge 3}$ and $\mathcal{C}_{\ge 2}$ are all distinct. Therefore we conclude that 
\begin{eqnarray*}
	|S(G)\setminus S(G_{w})\setminus \text{im}\,\Phi| &\ge& (|\mathcal{C}_{2}|+|\mathcal{C}_{3}|) +|\mathcal{C}_{1}|-|\mathcal{C}_{1}^{1}\cup \mathcal{C}_{1}^{4}/\sim|-|\mathcal{C}_{1}^{3}\cup \mathcal{C}_{1}^{4}/\sim|\\
&\ge& |C(G)\setminus C(G_{w})\setminus \text{im}\,\phi|.
\end{eqnarray*}
This proves the assertion.
\end{proof}

\begin{theorem}
	Let $G=(V,E)$ be a connected graph. Then for any vertex $w$ of $G$, we have  
\begin{equation*}
{\beth^{3}(G/w)\le \beth^{3}(G)}.
\end{equation*}
\end{theorem}

\begin{proof}
	If $G_{w}$ is a complete graph, then $G/w=G$ and the assertion follows. 
Let $H:=G_{w}-w$ and $K:=K_{w}-w$. We use inductin on $|V|$. If $|V|=1$, there is nothing to prove. Suppose $|V|>1$. First I claim that 
\begin{equation*}
	|S(G_{w})|-|S({K_{w}})|\ge |C(H)|-|C(K)|.
\end{equation*}
Write $G_{w}=H^{w}$ and $K_{w}=K^{w}$, the cone graphs with bases $H$ and $K$. Note that $H$ may not be connected; let $H_{1},\cdots,H_{k}$ be the components of $H$. But since $|V(H)|<|V(G)|$, the induction hypothesis applies to each component of $H$. Note that the induction hypothesis implies 
\begin{equation*}
	\Delta S \ge \Delta \beth^{2} \ge 0
\end{equation*}
by Theorem 3.2.1. Hence we apply the induction hypothesis to each component $H_{i}$ repeatedly until $H_{i}$ becomes a complete graph, while keeping the number of solids decreasing. Then we have corresponding complete graphs $K_{1},\cdots,K_{k}$ and $|S(K_{i})|\le |S(H_{i})|$. Now through nonedge contractions, identify all the complete graphs $K_{i}$ into the maximal one, say $K_{*}$. Note that $|S(K_{*})|=\max_{1\le i \le k}(|S(K_{i})|)\le \sum_{i=1}^{k}|S(H_{i})|=|S(H)|$. Then since vertex compressions are a composition of nonedge contractions, we get a sequence of nonedge contractions from $H$ to $K_{*}$ with $|S(K_{*})|\le |S(H)|$. Note that Proposition $4.2.1$ and $3.2.1$ implies 
\begin{eqnarray*}
	|S(H^{w})|&=&|S(H)|+|C(H)|\\
&=& |S(H)|+|C(H^{w})|-|E^{in}(H)|\\
&=& |S(H)|+|C(H^{w})|-|E^{in}(H^{w})|+|V^{in}(H^{w})|-1\\
&=& |S(H)|+\beth^{2}(H^{w})-1. 
\end{eqnarray*}
Then since $\beth^{2}(H^{w})\ge \beth^{2}(K_{w})$ by theorem 3.2.1, we have $|S(H^{w})|\ge |S(K^{w})|$. Now by the minimality of vertex compression, $K\subseteq K_{*}$ and hence $|S(K)|\le |S(K_{*})|$. Therefore we obtain
\begin{eqnarray*}
	|S(K_{w})|&=& |S(K)|+|C(K)|\\
	&\le & |S(K_{*})|+ |C(K)| \\
	&\le & |S(H)|+ |C(H)|-(|C(H)|-|C(K)|)\\
	&= & |S(G_{w})|-(|C(H)|-|C(K)|),
\end{eqnarray*}
proving our claim. Let us abbreviate $\Delta C(G-G_{w}):=|C(G)\setminus C(G_{w})|-|C(G/w)\setminus C(K_{w})|$. Note that 
\begin{eqnarray*}
\Delta C &=& |C(G_{w})|-|C(K_{w})|+\Delta C(G-G_{w}) \\
		&=& (|C(H)|-|C(K)|)+(|E(H)|-|E(K)|) +\Delta C(G-G_{w})
\end{eqnarray*}
and
\begin{eqnarray*}
\Delta E &\ge & |E(G_{w})|-|E(K_{w})| \\
		&=& (|E(H)|-|E(K)|)+(|V(H)|-|V(K)|) \\
		&=& (|E(H)|-|E(K)|)+\Delta V.
\end{eqnarray*}
Hence, with the claim and Lemma 4.3.1, we deduce that 
\begin{eqnarray*}
	\Delta S &=& (|S(G_{w})|-|S(K_{w})|)+|S(G)\setminus S(G_{w})|-|S(G/w)\setminus S(K_{w})|\\
&\ge & |C(H)|-|C(K)| + \Delta C(G-G_{w})\\
&=& \Delta C - (|E(H)|-|E(K)|)\\
&\ge & \Delta C - \Delta E +\Delta V. 
\end{eqnarray*}
This completes the induction. 
\end{proof}

\begin{corollary}
	Let $G=(V,E)$ be a connected graph. Then $\beth^{3}(K^{\chi(G)})\le \beth^{3}(G)$. That is, 
\begin{equation*}
\binom{\chi(G)}{4}-\binom{\chi(G)}{3}+\binom{\chi(G)}{2}-\binom{\chi(G)}{1} \le \beth^{3}(G).
\end{equation*}
\end{corollary}

Hence, we can estimate the hadwiger number and chromatic number by \textit{counting} the solids, induced cycles, edges and vertices. For example,  since $\beth^{3}(K^{9})=69$ and $\beth^{3}(K^{8})=34$, $\beth^{3}(G)<69$ yields $h(G),\chi(G)\le 8$ and $\beth^{3}(G)<34$ implies $h(G),\chi(G)\le 7$, and so on. 

\begin{proof}
	Follows from Theorem 4.3.1. The proof is similar to that of Corollary 2.2.1. 
\end{proof}

\begin{corollary}
	Let $G=(V,E)$ be a connected graph drawn in the plane and let $F$ be the number of regions of $G$. Then $|S(G)|-|C(G)|+F \ge 1$. 
\end{corollary}

\begin{proof}
	Note that the minimum of the function $\beth^{3}(K^{n})$ is $-1$ at $n=1,2,3,4$. Hence Corollary 4.3.1 yields $\beth^{3}(G)\ge -1$ for arbitrary graph $G$. Then the Euler characteristic for planar graph yields 
\begin{equation*}
	-1=-F+|E|-|V|+1\le |S(G)|-|C(G)|+|E|-|V|
\end{equation*}
and the assertion follows. 
\end{proof}

\begin{corollary}
	Let $G=(V,E)$ be a connected graph drawn in the plane and let $F$ be the number of regions of $G$. Then $\chi(G)\le 4$ if $|S(G)|-|C(G)|+F = 1$. 
\end{corollary}

\begin{proof}
	Proof follows from Corollary 4.3.1 and the Euler characteristic for planar graphs. That is, the hypothesis yields $\beth^{3}(G)\le -F+|E|-|V|+1=-1$ and $\beth^{3}(K^{n})\le -1$ yields $n\le 4$.
\end{proof}

Unfortunately, not all planar graphs satisfy the hypothesis of Corollary 4.3.3. The graph for the octahedron has $|C|=11$ and $|S|=6$, and obviously it has eight faces; hence $|S|-|C|+F=6-11+8=3$. 

\section{Applications}

\subsection{Upper bounds for $\chi(G)$ and $h(G)$ in terms of induced cycles and solids}

We know that the odd cycles are 3-critical graphs, and the absence of which as induced subgraphs yields 2-colourability. But is there any relation between the number of induced odd cycles and chromatic number? There is one quiet obvious relation: 
\begin{proposition}
	For a connected graph $G$, if $G$ has at most $\binom{n}{3}$ induced odd cycles, then 
$\chi(G)\le n$. 
\end{proposition}
To see this, consider a graph homomorphism $G\rightarrow K^{\chi(G)}$. Choose any triangle in $K^{\chi(G)}$ and let $H\subseteq G$ be the inverse image of it. Then it is obvious that $\chi(H)=3$, and hence $H$ has induced 3-critical graph, which is then an induced odd cycle of $G$. Since different choice of triangle in $K^{\chi(G)}$ yields different inverse image with distinct colour classes, there are induced odd cycles of $G$, one for each triangle of $K^{\chi(G)}$. This argument holds for any induced $r$-critical subgraphs, instead of induced 3-critical graphs(induced odd cycles). But can we do better? Does the number of such induced subgraphs also bounds hadwiger number as well? Note that a connected graph can have high hadwiger number without induced odd cycle; for instance, for any $r\in \mathbb{N}$, consider a subdividision of $K^{r}$ such that each edge becomes length 2 paths. Hence, number of induced odd cycle would not give an upper bound for hadwiger number. But the number of induced cycles, or even solids, does. We will prove following two theorems using the graph characteristic theory. 

\begin{theorem}
	Let $G$ be a connected graph. Then $|C(G)|<\binom{n}{3}$ implies $\max(\chi(G),h(G))<n$. 
\end{theorem}

\begin{theorem}
	Let $G$ be a connected graph. Then $|S(G)|<\binom{n}{4}$ implies $\max(\chi(G),h(G))<n$. 
\end{theorem}

What these results say is that a connected graph $G$ is 4-colourable without $K^{5}$ minor if $|C(G)|<10$, 5-colourable without $K^{6}$ minor if $|C(G)|<15$, and 6-colourable without $K^{7}$ minor if $|C(G)|<35$ and so on. On the other hand, $G$ is 4-colourable and has no $K^{5}$ minor if it contains less than 5 solids, 5-colourable with no $K^{6}$ minor if $|S(G)|<15$, and 6-colourable with no $K^{7}$ minor if $|S(G)|<35$ and so on. 

Note that these are the best possible upper bounds of $\chi(G)$ and $h(G)$ in terms of induced cycles and solid, since the complete graph $K^{n}$ satisfies those inequalities strictly. We only prove Theroem 5.1.2 since the proof for 5.1.1 is exactly the same. 

\begin{proof}[Proof of Theorem 5.1.2.]
Recall that Corollary 3.2.1 implies that $\chi(G)< n$ if $\beth^{2}(G)<\beth^{2}(K^{n})$. Similarly Corollary 4.3.1 yields that $\chi(G)<n$ if $\beth^{3}(G)<\beth^{3}(K^{n})$. Now assume $|S(G)|<\binom{n}{4}$. We may assume $\beth^{2}(G)\ge \beth^{2}(K^{n})$, since otherwise $\chi(G)<n$ as desired. Then we have 
\begin{eqnarray*}
\beth^{3}(G)&=&|S(G)|-\beth^{2}(G) \\ 
&<&\binom{n}{4}-\beth^{2}(K^{n}) = \beth^{3}(K^{n})
\end{eqnarray*}
and therefore we conclude $\chi(G)<n$ by Corollary 4.3.1. The proof for $h(G)$ is exactly the same from Corollary 3.1.1 and 4.2.1. 
\end{proof}

\subsection{Graph Characteristics and Hadwiger's Conjecture}

Hadwiger's conjecture states that a connected without $K^{t}$ minor is $t-1$ colourable. Hadwiger himself proved the conjecture for $t\ge 4$. The graphs without $K^{4}$ minor are the series-parallel graphs and their subgraphs. Hadwiger showed the conjecture for $t=4$ by showing that each graph of this type containes a vertex of degree $\ge 2$, which enables induction on the number of vertices. The conjecture for $t=5$ implies the four colour theorem, and Klaus Wagner proved that the conjecture for $t=5$ is actually equivalent to the four colour theorem, by showing that a graph without $K^{5}$ can be decomposed via clique-sums into planar graphs and a certain graph with 8 vertices, called the Wagner graph. Clique-sum is a graph operation that is a composition of disjoint union of two graphs followed by identification of the common complete graphs, which also can be thought as a graph theoretical analog of the connected sum in topology. It is clear that if a graph $G$ is obtained from two graphs $H_{1},H_{2}$ via clique-sum, then $\chi(G)=\max(\chi(H_{1}),\chi(H_{2}))$. Hence, the four-colourability of the whole graph then follows from that of the planar graphs and the Wagner graph. The conjecture for $t=6$ has shown by Robertson, Seymour and Thomas [9], using the four colour theorem. See [7] for more detailed survey on Hawiger conjecture. 

Let $\mathcal{H}$ be the collection of finite connected simple graphs for which Hadwiger's conjecture holds. We may call this class the \textit{Hadwiger class}. That is, $\mathcal{H}:=\{G=(V,E)\,|\, \chi(G)\le h(G)\}$. First we would like to see under what graph operations the Hadwiger class is closed, since  such operations will broaden our cunquered territory. One of such candidate is the \textit{Hajos operation}, which is defined as follows. If $G_{1},G_{2}$ are two connected graphs with vertices $x,y_{1},y_{2}$ such that $G_{1}\cap G_{2}=\{x\}$ and $xy_{1}\in E(G_{1})$ and $xy_{2}\in E(G_{2})$, then we define the graph $G_{1}*_{H}G_{2}$ obtained from $G_{1},G_{2}$ via the Hajos operation by 
\begin{equation*}
	G_{1}*_{H}G_{2} = (G_{1}\cup G_{2})-xy_{1}-xy_{2}+y_{1}y_{2}.
\end{equation*}

In the following arguments, we would not assume the four colour theroem. But we may use the fact that the Hadwiger class contains every graph $G$ such that $h(G)\le 3$, which is clear. 

\begin{proposition}
	The Hadwiger class $\mathcal{H}$ is closed under following graph operations:
\begin{description}
	\item{(a)} Deletion of a cut-edge
	\item{(b)} Clique-sum
	\item{(c)} Hajos operation
	\item{(d)} Subdivision
\end{description}
\end{proposition}

\begin{proof}
	
\begin{description}
	\item{(a)}  Let $e$ be a cut-edge of a graph $G\in \mathcal{H}$, and denote the two components of $G-e$ by $H_{1}$ and $H_{2}$. We may assume $h(G)\ge 3$. Then one has $h(G)=\max{h(H_{1}),h(H_{2})}$ and $\chi(G-e)=\max{\chi(H_{1}),\chi(H_{2})}\le \chi(G)$, from which we get $G-e\in \mathcal{H}$. 

	\item{(b)} Let $G_{1},G_{2}$ be graphs in the class $\mathcal{H}$ with subgraphs $K_{1},K_{2}$ which are isomorphic to $K^{r}$. Let $G$ be the graph obtained from $G_{1}$ and $G_{2}$ by identifying $K_{1}$ and $K_{2}$. Then colouring both graphs requires at least $r$ different colours. Now label the vertices of $K_{1}$ and $K_{2}$ by the numbers from $1$ to $r$, so that the two vertices of the same number are identified in the clique-sum. Now fix colourings of $G_{1}$ and $G_{2}$, and permute the colours so that each vertices of $K_{1}$ and $K_{2}$ gets the same colour as the label. Hence the proper colourings of $G_{1}$ and $G_{2}$ agrees on $K_{1}$ and $K_{2}$, which shows $\chi(G)\le \max(\chi(G_{1}),\chi(G_{2}))$. Since we need at least $\max(\chi(G_{1}),\chi(G_{2}))$ colours to colour $G$, we have $\chi(G)=\max(\chi(G_{1}),\chi(G_{2}))$. On the other hand, it is clear that $h(G)\ge \max(h(G_{1}),h(G_{2}))$. Hence we conclude $G\in \mathcal{H}$, since
\begin{equation*}
	\chi(G)=\max(\chi(G_{1}),\chi(G_{2}))\le \max(h(G_{1}),h(G_{2})) \le h(G).
\end{equation*}

	\item{(c)} Let $G:=G_{1}*_{H} G_{2}$ with the vertices $x,y_{1},y_{2}$ described in the definition of Hajos operation. If $xy_{1}$ is a cut-edge of $G_{1}$ and $xy_{2}$ is not a cut-edge of $G_{2}$, then $G$ can be obtained by first deleting the cutedge $xy_{1}$ from $G$ and then clique-summing the two componont to $G_{2}$ using the 1-cliques(single vertices) $x$ and $y_{2}$. Hence by $(a)$ and $(b)$, we have $G\in \mathcal{H}$. A similar argument applies to the case that both $xy_{1}$ and $xy_{2}$ are cut-edges. Now we suppose none of the two edges is a cut-edge.

If both $G_{1}$ and $G_{2}$ are 2-colourable, then $G$ is $3$-colourable. Since we know that the Hadwiger class contains all graphs with hadwiger number less than 4, it automatically follows that $G\in \mathcal{H}$. Now we may suppose $\chi(G_{2})\ge 3$. Since $xy_{1}$ is not a cut-edge of $G_{1}$, $G_{1}-xy_{1}$ is connected so that it is contractible to the single vertex $x$ via edge contractions. Once we contract $G_{1}-xy_{1}$ to $x$, we get a graph isomorphic to $G_{2}$ and hence $G_{2}$ is a minor of $G$. By symmetric argument, $G_{1}$ is also a minor of $G$. This shows $h(G)\ge \max{(h(G_{1}),h(G_{2}))}$. For the chromatic number, let $c_{1}$ be a proper colouring of $G_{1}$ with $c_{1}(x)=1$ and $c_{2}(y_{1})=2$. Then if $c_{2}$ is any proper colouring of $G_{2}$, then we may permute the colours so that $c_{2}(x)=1$ and $c_{2}(y_{2})=3$. We may assume $c_{2}$ is such a colouring. Then the colourings $c_{1}$ and $c_{2}$ induces a proper colouring $c$ of $G$, since $y_{1}$ and $y_{2}$ got differenc colour so that adding the edge $y_{1}y_{2}$ does not violate the colouring of $G_{1}\cup G_{2}$. This shows $\chi(G)\le \max(\chi(G_{1}),\chi(G_{2}))$, and therefore we have $G\in \mathcal{H}$.  

	\item{(d).} Let $G\in \mathcal{H}$ and let $G'$ be a subdivision of $G$. It suffices to show the assertion for $G'$ which is obtained from $G$ by subdividing a single edge $e=xy_{1}$, since any subdivision of $G$ can be obtained by successively subdividing the edges of $G$. Suppose $G'$ is obtained from $G$ by substituting the edge $e$ by a length $k$ path $P$.Let $C$ be a cycle of length $k$ such that $V(G)\cap V(C)=\{x\}$. Let $y_{2}$ be a neighbor of $x$ in $C$. Then observe that $G'$ can be obtained from the Hajos operation of $G_{1}$ and $C$, that is, 
\begin{equation*}
	G'=(G\cup C)-xy_{1}-xy_{2}+y_{1}y_{2}.
\end{equation*}
Then (c) yields $G'\in \mathcal{H}$. This proves the assertion. 
\end{description}

\end{proof}

We have seen that the graphs with hadwiger number $\ge 6$ belongs to the Hadwiger class. Moreover, some special graphs such as the complete graphs, complete multipartite graphs and perfect graphs  also belong to this class. Also recall that Chordal graphs are the graphs which have no induced cycles of length at least 4. Seymour and Weaver [4] characterized the chordal graphs as the graphs that can be obtained by clique sums of complete graphs. Since complete graphs are in the Hadwiger class and the class is closed under the clique sum, it is automatic that every chordal graphs are in the Hadwiger class as well. In fact, more strongly, the chordal graphs are actually perfect graphs, as observed by Berge [5].  In addition to the already observed kind of graphs that belong to the Hadwiger class, we add some more, which we have found through our graph characteristic theory. 

\begin{theorem}
	The Hadwiger class $\mathcal{H}$ contains a graph $G$ if it satisfies $\beth^{i}(K^{h(G)})=\beth^{i}(G)$ for some $i=1,2,3$. 
\end{theorem}

\begin{proof}
	The hypothesis and Corollary 2.2.1, 3.2.1 and 4.3.1 yields that $\beth^{i}(K^{\chi(G)})\le \beth^{i}(K^{h(G)})$. Since $\beth^{i}(K^{t})$ is an increasing function in $t\in \mathbb{N}$, we obtain $\chi(G)\le h(G)$. 
\end{proof}

In order to characterize such graphs described in the assertion, we need to figure out the special edge contractions that keeps the value of the graph characteristic the same, and consider the graphs obtained from the complete graphs through the "inverse" of that special edge contractions. We shall leave it as a further research.

\begin{figure}[H!]
	\centering
		\includegraphics[width=15cm,height=22cm]{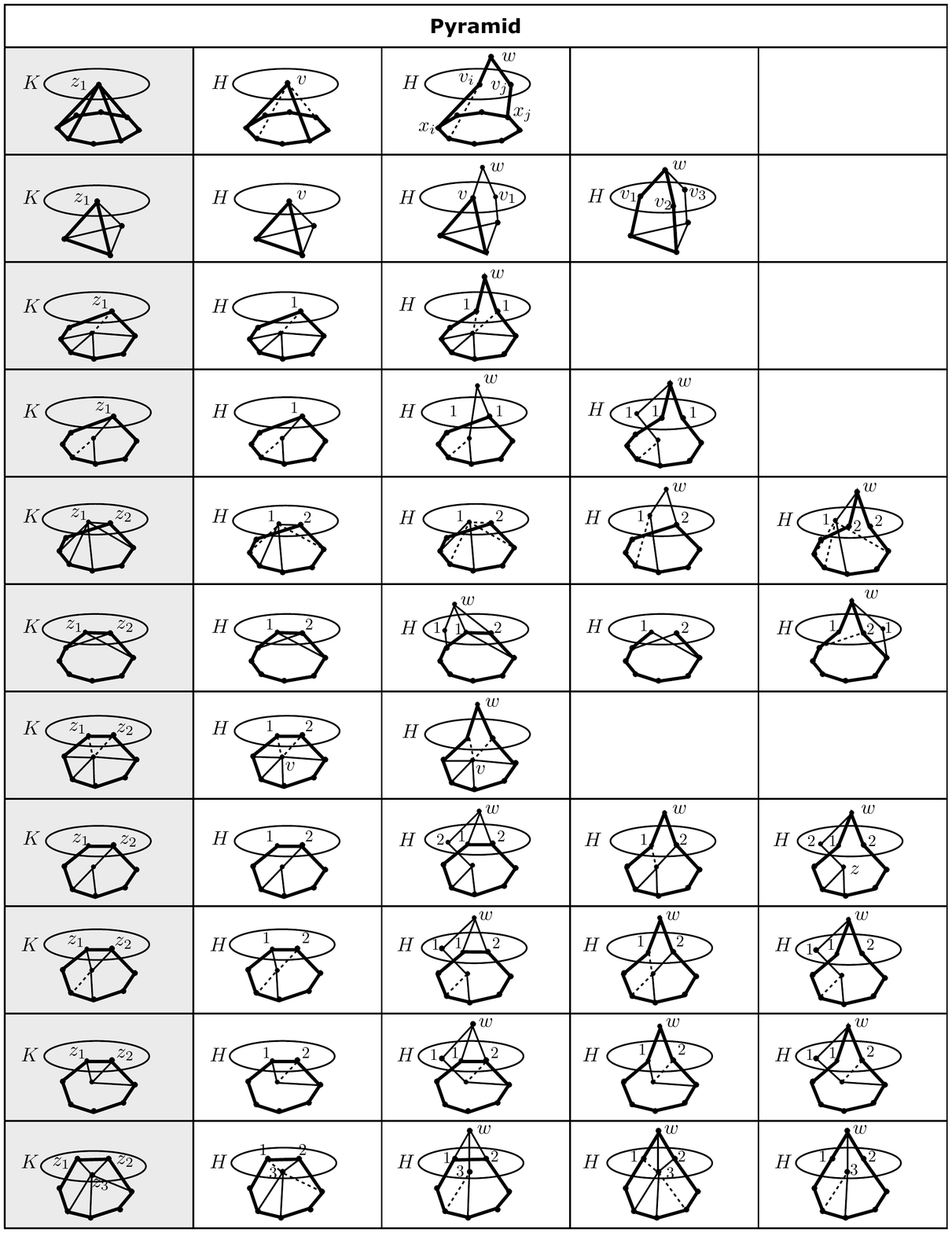}
\end{figure}
\begin{figure}[H!]
	\centering
		\includegraphics[width=15cm,height=22cm]{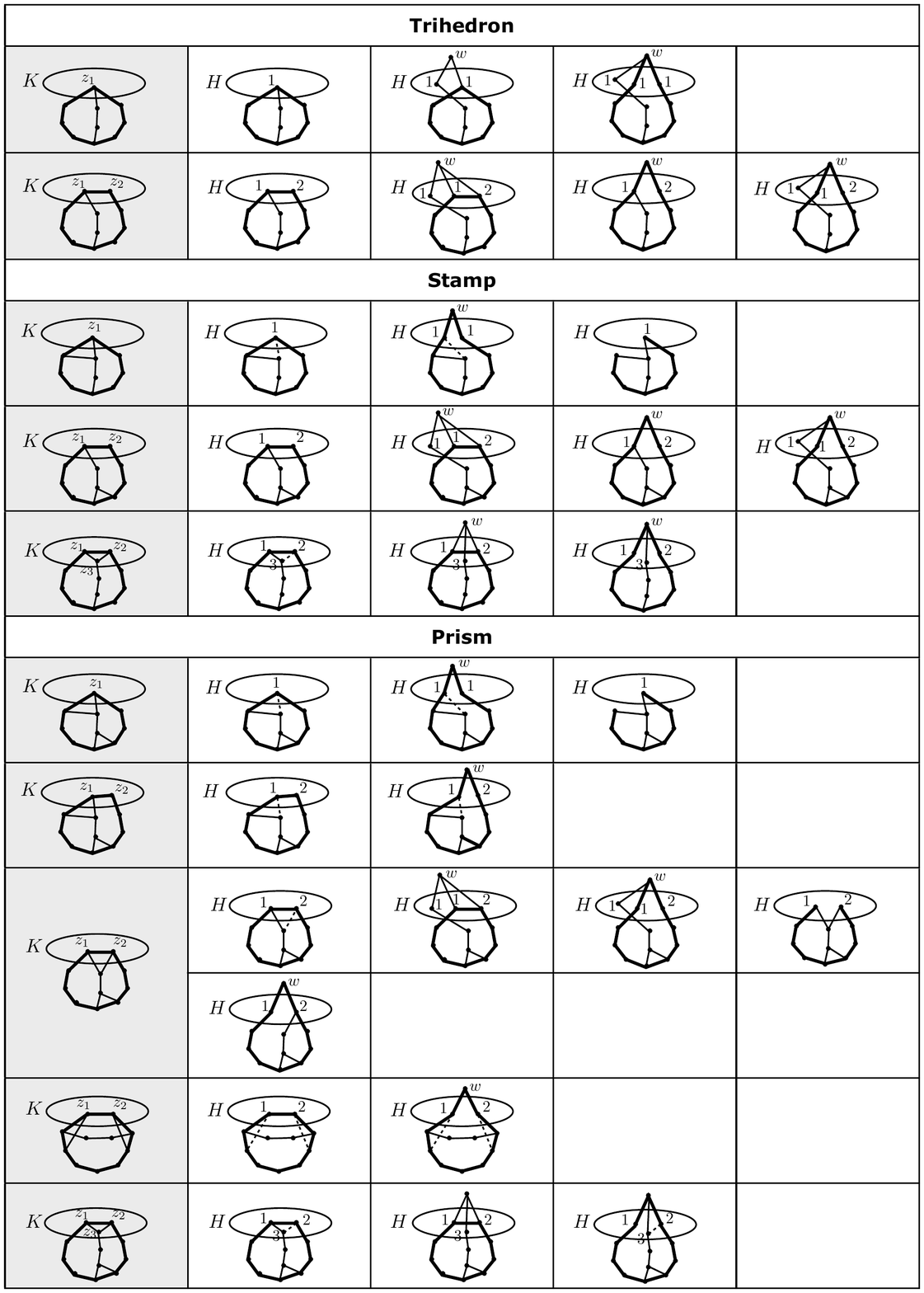}
\end{figure}

\newpage

\section*{Acknowledgements}

I especially appreciate Prof. Woong Kook of University of Rhode Island for his continued insightful comments, which helped me tremendously in many ways; I could clarify definitions and concepts, simplify the section 4.1, and correct many mistakes. Also I appreciate Prof. Hyuk Kim of Seoul National University encouraging my research and giving me advises that guided my research a lot.

\end{document}